\documentclass[reqno,11pt]{amsart}
\usepackage{amssymb,amsthm,amsmath,tocvsec2,bbm}
\usepackage[breaklinks=true]{hyperref}
\usepackage[tmargin=1.1in,bmargin=1.1in,rmargin=1.2in,lmargin=1.2in]{geometry}
\usepackage{enumerate,pgf,tikz,color}
\usetikzlibrary{arrows,shapes.geometric}
\allowdisplaybreaks

\theoremstyle{plain}
\newtheorem{theorem}{Theorem}[section]
\newtheorem{lemma}[theorem]{Lemma}
\newtheorem{proposition}[theorem]{Proposition}

\theoremstyle{definition}
\newtheorem{definition}[theorem]{Definition}
\newtheorem{remark}[theorem]{Remark}
\newtheorem{example}[theorem]{Example}
\newtheorem{question}[theorem]{Question}

\numberwithin{equation}{section}
\numberwithin{figure}{section}
\numberwithin{table}{section}

\newcommand{\R}{\mathbb{R}}
\newcommand{\bp}{\mathbb{P}}

\newcommand{\F}{\mathbb{F}}
\newcommand{\tangle}[1]{\langle #1 \rangle}
\renewcommand{\geq}{\geqslant}
\renewcommand{\leq}{\leqslant}

\newcommand{\bu}{\mathbf{u}}
\newcommand{\bv}{\mathbf{v}}
\newcommand{\bx}{\mathbf{x}}
\newcommand{\bz}{\mathbf{z}}

\newcommand{\cS}{\mathcal{S}}

\begin{document}

\title[Entrywise calculus, dimension-free positivity preservers, sphere
packings]{The entrywise calculus and dimension-free positivity
preservers, with an Appendix on sphere packings}

\author{Apoorva Khare}
\address[A.~Khare]{Department of Mathematics, Indian Institute of
Science, Bangalore 560012, India and Analysis \& Probability Research
Group, Bangalore 560012, India}
\email{\tt khare@iisc.ac.in}

\date{\today}

\begin{abstract}
We present an overview of a classical theme in analysis and matrix
positivity: the question of which functions preserve positive
semidefiniteness when applied entrywise. In addition to drawing the
attention of experts such as Schoenberg, Rudin, and Loewner, the subject
has attracted renewed attention owing to its connections to various
applied fields and techniques. In this survey we will focus mainly on the
question of preserving positivity in all dimensions. Connections to
distance geometry and metric embeddings, positive definite sequences and
functions, Fourier analysis, applications and covariance estimation,
Schur polynomials, and finite fields will be discussed.

The Appendix contains a mini-survey of sphere packings, kissing numbers,
and their ``lattice'' versions. This part overlaps with the rest of the
article via Schoenberg's classification of the positive definite
functions on spheres, aka dimension-free entrywise positivity preservers
with a rank constraint -- applied via Delsarte's linear programming
method.
\end{abstract}

\keywords{positive semidefinite matrix,
entrywise calculus,
Toeplitz matrix,
Hankel matrix,
absolutely monotonic function,
metric geometry,
spherical embedding,
positive definite function,
two-point homogeneous space,
Schur polynomials,
symmetric function identities,
sample covariance,
covariance estimation,
finite fields; \quad
sphere packing, packing density, lattice packing, $E_8$ lattice,
Leech lattice, kissing number, Hermite constant, spherical harmonics,
addition theorem, Gegenbauer polynomial, Chebyshev polynomial, spherical
code}

\subjclass[msc2020]{
15-02, 
26-02, 
52-02, 
52-03; 
15B48, 
51F99, 
15B48, 
42A70, 
15A15, 
15A45, 
15A83, 
62J10, 
47B34, 
42A82, 
43A35, 
51K05, 
30L05, 
52C17, 
33C55, 
11H31, 
94B27} 

\maketitle

\settocdepth{subsubsection}
\tableofcontents

\section{Introducing positivity preservers}

The goal of this article is to survey a foundational result in matrix
analysis, whose origins can be traced back to P\'olya and Szeg\H{o}
exactly one hundred years ago (following Schur). This result, originally
proved by Schoenberg (and then Rudin and many others), continues to yield
connections to active areas of mathematics and applied fields.

The result in question combines two evergreen ingredients in mathematics:
preserver problems and positive matrices. The notion of positivity is as
old as mathematics -- starting with counting and measuring. More
pertinently, positivity of real symmetric matrices occurs at least as
early as the second partial derivative test for local minima (via the
Hessian matrix). On the complex side, an early occurrence of positive
semidefinite matrices is in Pick and Nevanlinna's solutions of their
eponymous interpolation problem (1910s).

Recall that a Hermitian matrix $A \in \mathbb{C}^{n \times n}$ is said to
be \textit{positive definite} if the associated quadratic form $Q(x) :=
x^* A x$, $x \in \mathbb{C}^n$ is positive definite (this notation has
appeared as early as~{1868} in~\cite{Smith} -- in connection with the
existence of the $E_8$ lattice; see Section~\ref{Slattices} in the
Appendix). The non-strict relaxation of this condition is that of
\textit{positive semidefiniteness} of a Hermitian matrix $A$:
\begin{equation}\label{Epsd}
x^* A x = \tangle{Ax, x} \geq 0\ \forall x \in \mathbb{C}^n.
\end{equation}

Classical results by Sylvester and others provide numerous
characterizations of this notion:

\begin{theorem}\label{Tbasics}
The following are equivalent for a complex (resp.\ real) Hermitian matrix
$A_{n \times n}$:
\begin{enumerate}
\item $A$ is positive semidefinite (henceforth termed \textbf{psd}, or
simply {\em positive}): $x^* A x \geq 0\ \forall x \in \mathbb{C}^n$
(resp.\ $x \in \mathbb{R}^n$).

\item The eigenvalues of $A$ are all in $[0,\infty)$.

\item $A = B^* B$ for some 
$B \in \mathbb{C}^{n \times n}$ (resp.\ $B \in \mathbb{R}^{n \times n}$).

\item There exist vectors in $\mathbb{C}^n$ (resp.\ $\mathbb{R}^n$), say
$\bx_1, \dots, \bx_n$, such that $A$ is their {\em Gram matrix}:
$a_{ij} = \tangle{\bx_i, \bx_j}$ for all $1 \leq i,j \leq n$.

\item (Sylvester's criterion.) The principal minor $\det A_{I \times I}$
is nonnegative$,$ for all $I \subseteq [n] := \{ 1, \dots, n \}$.
\end{enumerate}
\end{theorem}

The following fact is also standard.

\begin{lemma}\label{Lbasics}
Let $A$ be the Gram matrix of any finite set of vectors drawn from
$\R^r$. Then $A$ has rank at most $r$, with equality if and only if the
vectors span $\R^r$.
\end{lemma}

In this survey, the \textbf{question of interest} is to understand the
functions that are applied to matrices and preserve positivity.
Understanding preservers of mathematical structures is an age-old
question; for example, in matrix theory one of the first such results is
by Frobenius, who in~{1897} classified all determinant-preserving linear
maps on matrix algebras~\cite{Frob}. Just as another example:
Marcus~\cite{Marcus59} and Russo--Dye~\cite{Russo-Dye66} classified
linear preservers of the unitary group in $\mathbb{C}^{n \times n}$ and
in general $C^*$-algebras, respectively. However, even after a century of
work, basic questions in preservers of positive matrices remain
unanswered. As a first example, the linear preservers of positive
semidefinite matrices are not fully determined; for this and other
aspects of the area, see the survey articles~\cite{GLS2000,Li-Pierce01}
and the monograph~\cite{Molnar07}. (The goal of this survey is to discuss
nonlinear preservers, and we will mention another open question below.)

\begin{definition}
Returning to the question of interest -- and removing the linearity
constraint -- there are two natural ways in which a function acts on a
Hermitian matrix $A_{n \times n} = U^* D U$ (by the spectral theorem),
with $D = {\rm diag}(\lambda_1, \dots \lambda_n)$:
\begin{itemize}
\item On its spectrum, via the \textit{functional calculus}: $f(A) := U^*
f(D) U$, where $f(D) := {\rm diag}(f(\lambda_1), \dots, f(\lambda_n))$;
and
\item On its entries, via the \textit{entrywise calculus}: $f[A] :=
(f(a_{ij}))_{i,j=1}^n$.
\end{itemize}
There is also a second notion of positivity:
\begin{equation}
A = (a_{ij}) \in  \mathbb{C}^{n \times n} \text{ is } entrywise\
nonnegative \text{ if } a_{ij} \in [0,\infty) \ \forall i,j \in [1,n].
\end{equation}
\end{definition}

While the functional calculus is the more well-known mechanism, this work
will mainly focus on the entrywise calculus. From above, we find four
ways in which a function may act on matrix spaces and preserve
positivity:
\begin{enumerate}
\item $f(A)$ is psd if $A$ is psd.

\item $f[A]$ is entrywise nonnegative if $A$ is so.

\item $f(A)$ is entrywise nonnegative if $A$ is so.

\item $f[A]$ is psd if $A$ is psd.
\end{enumerate}\smallskip

\noindent \textbf{Notation.} Unless otherwise declared, we will
henceforth focus on real symmetric matrices, and real-valued functions
acting on them.\medskip

Now the first two of the four classifications above are easy:

\begin{proposition}
The functions satisfying conditions~(1) or~(2) above, are precisely the
functions $f : [0,\infty) \to [0,\infty)$.
\end{proposition}

The third was worked out by Hansen, for all positive matrices with
entries in a symmetric interval (for completeness, we also mention the
related work~\cite{bharali} of Bharali--Holtz):

\begin{theorem}[{\cite[Theorem~3.3(1)]{Hansen92}}]
Fix $0 < \rho \leq \infty$, and let a function $f : (-\rho,\rho) \to \R$.
The following are equivalent.
\begin{enumerate}
\item If $A$ is any real symmetric matrix (of any dimension) with
spectrum in $(-\rho,\rho)$ and nonnegative entries, then $f(A)$ has
nonnegative entries.

\item $f$ is given on $(-\rho,\rho)$ by a convergent power series
$\sum_{k \geq 0} c_k x^k$ with nonnegative coefficients: $c_k \geq 0\
\forall k \geq 0$.
\end{enumerate}
\end{theorem}

This leaves the fourth and final question, which is the subject of this
article:

\begin{question}
\textit{Which entrywise maps preserve positive semidefiniteness (of
matrices of all dimensions)?}
\end{question}

\section{Schoenberg's theorem and its (classical) variants}

We now embark on the study of the question above. Our journey begins
exactly one hundred years ago, which is when the question was asked, in
the well-known 1925 book \cite{polya-szego} of P\'olya and Szeg\H{o}. The
authors also supplied a large class of functions that are entrywise
positivity preservers -- the reason behind this is a celebrated 1911
result of Schur:

\begin{theorem}[{Schur product theorem,
\cite[Satz~VII]{Schur1911}}]\label{Tschur}
If $n \geq 1$, and two $n \times n$ matrices $A,B$ are positive
semidefinite, then so is their {\em Schur/entrywise product}
\begin{equation}
A \circ B := (a_{ij} b_{ij})_{i,j=1}^n.
\end{equation}
\end{theorem}

Among other things, this theorem is useful in proving that the closed
convex cone of positive definite kernels (defined below) on $X \times X$
for any set $X$ is moreover closed under multiplication. 

\begin{proof}
There are several proofs of this result: for instance, $A \circ B$ is a
principal submatrix of the Kronecker product $A \otimes B$, which is
positive semidefinite. Alternately, write
\[
A = \sum_{i \geq 1} \lambda_i v_i v_i^*, \qquad B = \sum_{j \geq 1} \mu_j
u_j u_j^*, \qquad \text{with all } \lambda_i, \mu_j \geq 0
\]
by the spectral theorem and Theorem~\ref{Tbasics}. As the entrywise/Schur
product is bilinear, $\displaystyle A \circ B = \sum_{i,j \geq 1}
\lambda_i \mu_j (v_i \circ u_j) (v_i \circ u_j)^*$, and this is positive
semidefinite from first principles.
\end{proof}

The Schur product theorem helps find entrywise positivity
preservers as follows. It is clear from the definition~\eqref{Epsd} that
the positive semidefinite (psd) matrices form a closed convex cone: they
are stable under sums, positive dilations, and entrywise limits. In
addition, by Theorem~\ref{Tschur} they are also closed under Schur
products. Thus, the set of entrywise maps preserving positivity is also
closed under these operations. In addition, this set contains the
functions $f(x) \equiv 1, x$ -- since the latter leaves each psd matrix
unchanged, while the former sends it entrywise to the all-ones matrix
${\bf 1}_{n \times n} := (1)_{i,j=1}^n$, and this rank-one matrix is psd.
Thus, the closure of $\{ 1, x \}$ under sums, positive dilations,
products, and limits preserves positivity of matrices of all sizes. This
closure is precisely the set of convergent power series with nonnegative
coefficients, and this was the 1925 observation of P\'olya--Szeg\H{o}:

\begin{definition}
Given a subset $I \subseteq \mathbb{C}$, define $\bp_n(I)$ to be the set
of $n \times n$ positive semidefinite matrices with entries in $I$.
\end{definition}

\begin{theorem}[{\cite[Problem~37]{polya-szego}}]\label{Tps}
Let $I \subseteq \R$ be an interval, and $f(x) = \sum_{k=0}^\infty c_k
x^k$ be a power series convergent on $I$, with all $c_k \in [0,\infty)$.
Then the entrywise map $f[-]$ sends $\bp_n(I)$ to $\bp_n(\R)$ for all $n
\geq 1$.
\end{theorem}

\subsection{Matrices with entries in $[-1,1]$}

After stating their result, P\'olya--Szeg\H{o} asked if there are other
preservers. This was answered by Schoenberg (a student of Schur and
Sanielevici) 17 years later -- for continuous functions. The next result
-- and its later variants stated below -- can be considered as
collectively forming a deep converse to the Schur product theorem.

\begin{theorem}[{\cite[Theorem~2]{Schoenberg42}}]\label{Tschoenberg}
Let $I = [-1,1]$ and let $f : I \to \R$ be continuous. The following are
equivalent.
\begin{enumerate}
\item The entrywise map $f[-]$ sends $\bp_n(I)$ to $\bp_n(\R)$ for all $n
\geq 1$.

\item $f$ is given on $I$ by a convergent power series $\sum_{k=0}^\infty
c_k x^k$ with all $c_k \geq 0$.
\end{enumerate}
\end{theorem}

\begin{definition}
Note that each such power series $f(x) = \sum_{k \geq 0} c_k x^k$ is
infinitely differentiable on $I' = (0,1)$ and satisfies: $f^{(k)} \geq 0$
on $I'$. Such a function is said to be \textit{absolutely
monotone/monotonic} on $I'$.
\end{definition}

\begin{remark}\label{Rreduce}
The easy implication is $(2) \implies (1)$, and this is precisely
Theorem~\ref{Tps}. The hard part is $(1) \implies (2)$, which was
Schoenberg's contribution. In fact, Schoenberg assumed positivity
preservation on an even smaller set -- the matrices in $\bp_n([-1,1])$
with all diagonal entries $1$, aka \textit{correlation matrices} -- and
deduced absolute monotonicity.
(From this, one can deduce the power series representation -- and in
particular, real analyticity -- using a 1929 theorem of
Bernstein~\cite{Bernstein}.)
This effort to ``reduce the test set'' will recur in this section.
\end{remark}

Let $C_1$ denote the multiplicative closed convex cone of entrywise
positivity preservers of $\bigcup_{n \geq 1} \bp_n([-1,1])$. In a sense,
the countably many (rescaled) monomials $\R_{\geq 0} x^k$ are the
``extreme rays'' of $C_1$ that are continuous.

It is natural to ask what happens if the continuity assumption is
removed. In this case, there are two other extreme rays that are
discontinuous -- but only at the endpoints -- and they are generated by
the following functions on $[-1,1]$:
\[
f_+(x) := \lim_{n \to \infty} x^{2n} = \mathbbm{1}(x = \pm 1), \qquad
f_-(x) := \lim_{n \to \infty} x^{2n+1} = \mathbbm{1}(x=1) -
\mathbbm{1}(x=-1),
\]
where $\mathbbm{1}(E)$ denotes the indicator of an event/statement $E$.
Thus, nonnegative linear combinations of $f_+, f_-$, and the functions
in Schoenberg's theorem~\ref{Tschoenberg} are indeed (possibly
discontinuous) preservers. A natural question is if there are no others,
and this was affirmatively answered in~1978 by Christensen and Ressel:

\begin{theorem}[{\cite[Theorems~1,2]{ChrRes1}}]
Let $I = [-1,1]$ and let $f : I \to \R$. The following are equivalent.
\begin{enumerate}
\item The entrywise map $f[-]$ sends $\bp_n(I)$ to $\bp_n(\R)$ for all $n
\geq 1$.

\item The function $f$ is equal to a convergent power series plus
two other terms:
\[
f(x) = \sum_{k=0}^\infty c_k x^k + c_{-1} \left( \mathbbm{1}(x=1) -
\mathbbm{1}(x=-1) \right) + c_{-2} \mathbbm{1}(x = \pm 1), \qquad x \in
I,
\]
with $c_k \geq 0$ for all $k \geq -2$ and $\sum_{k \geq -2} c_k <
\infty$.
\end{enumerate}
\end{theorem}

The methodologies employed in showing the above results are different:
Schoenberg used spherical integrals and ultraspherical polynomials to
prove his result, while Christensen--Ressel's proof was
convexity-theoretic, involving Choquet theory and Bauer simplices.

\subsection{Positive definite sequences, Toeplitz matrices on the circle,
and Rudin's result}

If one removes the endpoints from the domain, then the assumption of
continuity may also be dispensed with, with greater ease. This was first
achieved by Rudin in 1959 -- here is a reformulation of his result:

\begin{theorem}[{\cite[Theorems~I,IV]{Rudin59}}]\label{Trudin}
Suppose $I = (-\rho,\rho)$ where $0 < \rho \leq \infty$, and $f : I \to
\R$. The following are equivalent.
\begin{enumerate}
\item The entrywise map $f[-]$ sends $\bp_n(I)$ to $\bp_n(\R)$ for all $n
\geq 1$.

\item If $A \in \bp_n(I)$ is Toeplitz of rank at most $3$, then $f[A] \in
\bp_n(\R)$.

\item $f$ is given on $I$ by a convergent power series $\sum_{k=0}^\infty
c_k x^k$ with all $c_k \geq 0$.
\end{enumerate}
\end{theorem}

Note the significant reduction of the test set in part~(2), from all
positive matrices of all sizes in part~(1) -- in the spirit of
Remark~\ref{Rreduce}.

Rudin was studying preservers of \textit{positive definite
sequences}\footnote{This is the second occurrence, after ``matrices'', of
the phrase ``positive definite'', and we will see one more, which
naturally leads to entrywise transforms.},
and we now take some time to motivate these, from complex function
theory. In 1907, Carath\'eodory published a solution to the following
question:

\textit{Characterize all analytic functions $f$ such that $f(0) = 1$ and
$f$ maps the unit disk $D(0,1)$ into the right half-plane $\Re(z)>0$.}

His solution \cite{Caratheodory} was that if one writes $f(z) = 1 +
\sum_{k=0}^\infty (a_k + i b_k) z^k$, then the above condition holds if
and only if for each $n \geq 1$, the point $(a_1, b_1, \dots, a_n, b_n)
\in \R^{2n}$ lies in the convex hull of the curve
\begin{equation}\label{Ecaratheodory}
\{ (2 \cos \theta, -2 \sin \theta, 2 \cos 2 \theta, \dots, 2 \cos n
\theta, -2 \sin n \theta) \, : \, 0 \leq \theta \leq 2 \pi \}.
\end{equation}

In 1911, Toeplitz observed~\cite{Toeplitz} that the
constraints~\eqref{Ecaratheodory} can be rephrased algebraically, in
terms of the positivity of certain related Hermitian quadratic forms for
all $n \geq 0$:
\[
\sum_{i,j=1}^n c_{i-j} z_i \overline{z_j} \geq 0\ \forall z = (z_1,
\dots, z_n) \in \mathbb{C}^n, \qquad \text{where } c_0 = 2, \ c_k = a_k -
i b_k, \, c_{-k} = \overline{c_k} \, (k>0).
\]
In other words, the semi-infinite matrix
\begin{equation}\label{Etoeplitz}
T = (t_{ij}) := (c_{i-j})_{i,j \geq 0} \text{ is positive semidefinite.}
\end{equation}
Moreover, Toeplitz's matrix $T$ here has the
property that the entries along any ``diagonal line'' (i.e., parallel to
the main diagonal) are all equal -- a structure that is now called a
\textit{Toeplitz matrix}.

Thus, Carath\'eodory's solution is equivalent to the notion of a positive
definite sequence, and it was the preservers of these that Rudin was
classifying in~\cite{Rudin59}. Rudin's motivations come from Fourier
analysis, whose connection to positive definite sequences was established
at the same time as Toeplitz's work. Indeed, in 1911, Herglotz also
published (independently) the work~\cite{Herglotz}, which showed the
equivalence of Toeplitz's conditions to the \textit{trigonometric moment
problem}:
to characterize all sequences $c_n \in \mathbb{C}$ for $n \in
\mathbb{Z}$, for which there exists a nonnegative measure $\mu$ on
$[-\pi,\pi]$ such that the $c_n$ are its Fourier--Stieltjes coefficients
$\int_{-\pi}^\pi e^{ - i n \theta} d \mu(\theta)$ for all $n$.

Herglotz showed that the answer is precisely the positivity
condition~\eqref{Etoeplitz}. Thus, positive definite sequences are
precisely the Fourier--Stieltjes coefficients of nonnegative measures
$\mu$ on $S^1 \subset \mathbb{C}$:

\begin{theorem}[\cite{Herglotz}]\label{Therglotz}
A complex sequence $(c_n)_{n \in \mathbb{Z}}$ is the Fourier--Stieltjes
coefficient-sequence of a nonnegative measure on $S^1$ if and only if the
Toeplitz matrix $T = (c_{i-j})_{i,j \geq 0}$ is positive semidefinite --
in other words, $c : \mathbb{Z} \to \mathbb{C}$ is a {\em positive
definite function} (defined below).
\end{theorem}

We add for completeness that alongside the solution to the above function
theory problem by Carath\'eodory, and the two equivalent conditions by
Toeplitz and Herglotz, comes yet another classical equivalent condition.
This is the famous Herglotz--Riesz (integral) representation theorem for
the aforementioned analytic maps, proved by both authors independently
in~1911. We state a more general version, wherein $f(0)$ need not equal
$1$:

\begin{theorem}[\cite{Herglotz,Riesz}]
A function $f(z) = u(z) + i v(z)$ is analytic on $D(0,1)$ with image in
the closed right half-plane $\Re(z) \geq 0$, if and only if there exists
a finite positive measure $\mu$ on $[0,2\pi]$ such that
\[
f(z) = i \cdot v(0) + \int_0^{2\pi} \frac{e^{i\theta}+z}{e^{i\theta}-z}\
d \mu(\theta).
\]
\end{theorem}

Returning to Rudin: he came to his question about preservers of such
sequences in the context of Fourier analysis. He was considering
functions operating on spaces of Fourier transforms of $L^1$ functions on
Locally Compact Abelian groups $G$ (such $G$ are abbreviated \textit{LCA}
groups), or of measures on $G$. Rudin studied the torus $G = S^1$, while
Kahane and Katznelson were studying similar questions on its dual group
$\mathbb{Z}$. The three authors then proved in 1959 with Helson, a
``converse Wiener--Levy theorem'' in~\cite{HKKR}. In the same year, Rudin
published his related variant of Schoenberg's theorem~\cite{Rudin59}.

\subsection{Open intervals; Hankel matrices}

We return to the story of entrywise preservers. First note that the rank
of the positive Toeplitz matrix in the trigonometric moment problem
above, corresponds to the size of the support of the measure $\mu$. Thus,
Rudin's Theorem~\ref{Trudin} shows that working with at most three-point
measures suffices to recover real analyticity and absolute monotonicity.
This connection between supports of measures and rank constraints of test
matrices resurfaced in positivity preservers very recently, and is
described below.

Two decades after Rudin's work, a variant of the above results was shown
for matrices with strictly positive entries. This was by Vasudeva in
1979:

\begin{theorem}[{\cite[Theorem~6]{vasudeva79}}]
For $I = (0,\infty)$, the two assertions of Schoenberg's
theorem~\ref{Tschoenberg} are again equivalent.
\end{theorem}

Vasudeva's theorem was strengthened twofold very recently: in 2022,
Belton--Guillot--Khare--Putinar obtained the same conclusion of absolute
monotonicity from hypotheses that were significantly weaker in two ways.
First, the domain was changed to $(0,\rho)$ for any $0 < \rho \leq
\infty$; and the test set in each dimension was once again reduced, this
time to \textit{Hankel} matrices of rank at most $2$.

\begin{theorem}[{\cite[Theorem~9.6 and
(proof of) Proposition~8.1]{BGKP-hankel}}]\label{Tbgkp-1sided}
Suppose $I = (0,\rho)$ or $[0,\rho)$ where $0 < \rho \leq \infty$, and $f
: I \to \R$. The following are equivalent.
\begin{enumerate}
\item The entrywise map $f[-]$ sends $\bp_n(I)$ to $\bp_n(\R)$ for all $n
\geq 1$.

\item If $A \in \bp_n(I)$ is Hankel of rank at most $2$, then $f[A] \in
\bp_n(\R)$.

\item $f$ is given on $I$ by a convergent power series $\sum_{k=0}^\infty
c_k x^k$ with all $c_k \geq 0$.
\end{enumerate}
\end{theorem}

Below, we will state the multivariate version of this result. Also note
that parallel to Rudin's approach (via Herglotz), in~\cite{BGKP-hankel}
the authors considered transforms of discrete data obtained from a
positive measure $\mu$ on the real line: \textit{moment sequences}. In
other words, a function $f$ sends the $k$th moment of a positive measure
$\mu$ on the line, to the $k$th moment of another positive measure
$\sigma_\mu$. Once again, the rank of the test matrices is governed by
the support set of $\mu$, and it turns out that one- and two-point test
measures suffice to deduce absolute monotonicity.

The two-sided version of Theorem~\ref{Tbgkp-1sided}, parallel to Rudin's
characterization above, was also shown in~\cite{BGKP-hankel}. Once again,
the test set can be greatly reduced in each dimension, from all
positive matrices to low rank Hankel matrices.

\begin{theorem}[{\cite[Corollary~6.2]{BGKP-hankel}}]\label{Tbgkp-2sided}
Suppose $0 < \rho \leq \infty$ and $I = (-\rho,\rho)$. Given $f : I \to
\R$, the following are equivalent.
\begin{enumerate}
\item The entrywise map $f[-]$ sends $\bp_n(I)$ to $\bp_n(\R)$ for all $n
\geq 1$.

\item If $A \in \bp_n(I)$ is Hankel of rank at most $3$, then $f[A] \in
\bp_n(\R)$.

\item $f$ is given on $I$ by a convergent power series $\sum_{k=0}^\infty
c_k x^k$ with all $c_k \geq 0$.
\end{enumerate}
\end{theorem}

Notice that all of the above results are ``dimension-free'', in that the
test set in them consists of matrices of unbounded size. It is
interesting that the proof of the last two results above uses a result in
\textit{fixed} dimension -- see Theorem~\ref{Tloewner}. A full account of
both of these proofs -- which are rather accessible to develop ``almost
from scratch'' -- as well as details of the results in multiple sections
below, can also be found in the recent monograph~\cite{AK-book}.
For additional connections and results, see the
survey~\cite{BGKP-survey1}.

As a parting note in the real case, understanding positivity preservers
immediately leads to characterizing the entrywise maps preserving
\textit{monotonicity}:

\begin{definition}
The \textit{Loewner ordering} on real symmetric $n \times n$ matrices is:
$A \geq B$ if $A-B \in \bp_n(\R)$. Now given $I \subseteq \R$, a function
$f : I \to \R$ is said to be
(a)~\textit{Loewner positive on $\bp_n(I)$} if $f[A] \geq 0$ whenever $A
\in \bp_n(I)$ (i.e., $A \geq 0$); and
(b)~\textit{Loewner monotone on $\bp_n(I)$} if $f[A] \geq f[B]$ whenever
$A \geq B \geq 0$.
\end{definition}

In this language, the above ``Schoenberg-type theorems'' classify the
Loewner positive maps. Setting $B=0$, it is also clear that if $0 \in I$
and $f$ is Loewner monotone, then $f - f(0)$ is Loewner positive. In fact
the reverse implication also holds -- and is easy to show using the Schur
product theorem, once we know the Schoenberg--Rudin theorem above. Thus,
we have:

\begin{theorem}[{\cite[Theorem~19.2]{AK-book}}]
Suppose $0 < \rho \leq \infty$ and $I = (-\rho,\rho)$. The following are equivalent for an arbitrary map $f : I \to \R$:
\begin{enumerate}
\item $f[-]$ is Loewner monotone on $\bp_n(I)$ for all $n \geq 1$.

\item $f[-]$ is Loewner monotone on the Hankel matrices in $\bp_n(I)$ of
rank at most $3$, for all $n \geq 1$.

\item $f$ is given on $I$ by a power series $\sum_{k=0}^\infty c_k x^k$,
with $c_k \geq 0$ for all $k>0$ (and any $c_0 \in \R$).
\end{enumerate}
\end{theorem}

The reader can compare and contrast this to Loewner's celebrated
characterization of matrix monotone maps in the functional
calculus~\cite{Loewner34,Simon}.

\subsection{Complex domains}

We conclude this section by discussing two results proved in the
literature for positivity preserving maps acting on complex Hermitian
matrices. Following his proof for preservers of real positive matrices,
Rudin~\cite{Rudin59} made an observation parallel to that of
P\'olya--Szeg\H{o}, for complex matrices. Namely, Rudin observed that the
conjugation map $z \mapsto \overline{z}$ also preserves positive
semidefiniteness. Since the preservers form a multiplicatively closed
convex cone, it follows that the entrywise functions
\[
z \mapsto z^j \overline{z}^k \qquad ( j, k \geq 0 )
\]
also preserve positivity on complex matrices of all sizes. Again taking
nonnegative linear combinations, followed by limits, yields a large
family of preservers; and Rudin conjectured in his 1959
work~\cite{Rudin59} that there are no others. This was shown by Herz soon
after, in 1963:

\begin{theorem}[{\cite[Th\'eor\`eme~1]{Herz63}}]\label{Therz}
Denote by $D( 0, 1 )$ the open unit disk in $\mathbb{C}$, and say $f : D(
0, 1 ) \to \mathbb{C}$. The following are equivalent.
\begin{enumerate}
\item The entrywise map $f[-]$ sends $\bp_n (D( 0, 1))$ to
$\bp_n(\mathbb{C})$ for all $n \geq 1$.

\item The function $f$ is a convergent power series, of the form $f( z )
= \sum_{j, k \geq 0} c_{j k} z^j \overline{z}^k$ for $z \in D(0,1)$, with
$c_{j k} \geq 0\ \forall j, k \geq 0$. Moreover, such a representation
for $f$ is unique.
\end{enumerate}
\end{theorem}

A final result along these ``classical lines'' returns to the
``Schoenberg'' version of the complex setting, of functions on the closed
unit disk. The entrywise positivity preservers were classified in this
setting by Christensen--Ressel in~1982, under Schoenberg's assumption of
continuity:

\begin{theorem}[{\cite[Corollary~1]{ChrRes2}}]\label{Tchrres2}
Let $f : \overline{D}(0,1) \to \mathbb{C}$ be a continuous map on the
closed unit disk. Let $\mathcal{H}$ be an infinite-dimensional complex
Hilbert space, with unit sphere $S$. The following are equivalent.
\begin{enumerate}
\item $f$ is a ``positive definite kernel'' on $S$ -- i.e., for any
finite set of points $z_1, \dots, z_n \in S$, the matrix with $(i,j)$
entry $f(\tangle{z_i, z_j})$ is positive semidefinite.

\item As in the previous result: $f$ is a convergent power series of the
form $f( z ) = \sum_{j, k \geq 0} c_{j k} z^j \overline{z}^k$ for $z \in
\overline{D}(0,1)$, with $c_{j k} \geq 0\ \forall j, k \geq 0$. Such a
representation for $f$ is unique.
\end{enumerate}
As a consequence, these are also equivalent to:
\begin{enumerate}
\setcounter{enumi}{2}
\item $f[-] : \bp_n(\overline{D}(0,1)) \to \bp_n(\mathbb{C})$ for all $n
\geq 1$.
\end{enumerate}
\end{theorem}

\section{Schoenberg's motivations: distance geometry}

We now discuss the classical motivations of Schoenberg in arriving at
Theorem~\ref{Tschoenberg}, which is a seminal result that has spawned
much subsequent activity (not only in generalizing and refining it, but
in other, modern domains as well, as is mentioned below). Schoenberg was
motivated by the study of \textit{metric/distance geometry}. Indeed,
since the advent of Descartes in the 1600s, it has been the norm
to think of vectors in Euclidean space $\R^n$ as $n$-tuples of real
numbers, and of distances between them via the Pythagorean metric.
However, for almost two millennia before Descartes, geometry meant
working with Euclid's postulates and using points, lines, angles,
distances, etc.

In the early 20th century, there was a revival of this ``distance
geometry'' perspective. We single out the Vienna Circle of mathematicians
and philosophers, in which Karl Menger (a student of Hahn), Tarski, Hahn,
von Neumann, G\"odel, Taussky, and others met regularly to discuss
mathematics. A prevalent theme of their Kolloquium (1928--1936,
see~\cite{Menger}) was the study of metric spaces and of properties
intrinsic to them, such as curvature, homogeneity, and metric convexity.
Indeed, the foundations of metric space theory had been then-recently
established, due to works of Birkhoff, Fr\'echet, and Hausdorff among
others, and the Vienna Kolloquium's investigations led them to advances
not only in mathematics, but also in mathematical economics (von
Neumann's fixed point theorem -- the precursor to the one by Kakutani
that was later used by Nash -- is a case in point, as is the work of
Abraham Wald).

\subsection{Metric embeddings}

Let $X = \{ x_0, \dots, x_n \}$ be a finite set endowed with a metric $d
= d_X$. A seminal 1910 result by Fr\'echet states:

\begin{theorem}[\cite{Frechet0}]
Every metric space of size $n+1$ isometrically embeds into $(\R^n,
\ell_\infty)$ -- where $\ell_\infty(x,y) := \sup_i |x_i-y_i|$ denotes the
sup-norm.
\end{theorem}

\begin{remark}
In fact the embedding Fr\'echet provides is reminiscent of -- and the
precursor to -- the Kuratowski embedding~\cite{Kuratowski}. Fr\'echet's
result was subsequently improved by Witsenhausen~\cite{Witsenhausen} to
$(\R^{n-1}, \ell_\infty)$ if $n \geq 2$.
\end{remark}

It is now natural to ask which finite, or separable, metric spaces embed
into ``usual'' Euclidean space $\R^n$, or into their limit $\R^\infty =
\bigcup_{n \geq 1} \R^n$, or into its closure $\ell^2(\R)$. Following
characterizations by Menger~\cite{Menger0} and Fr\'echet~\cite{Frechet},
Schoenberg provided in~1935 a characterization that related metric
geometry and matrix positivity:

\begin{theorem}[{\cite[Theorem~1]{Schoenberg35}}]\label{Teuclidean}
Let a finite metric space $X = \{ x_0, \dots, x_n \}$, and write $d_{ij}
:= d_X(x_i,x_j)$ for $i,j=0,\dots,n$. Then $X$ embeds isometrically into
some Euclidean space (equivalently, into $\ell^2$) if and only if its
{\em modified Cayley--Menger matrix}
\[
CM'(X) := ( d_{i0}^2 + d_{j0}^2 - d_{ij}^2 )_{i,j=1}^n
\]
is positive semidefinite. Moreover, if this happens then the smallest
dimensional space $\R^r$ into which $X$ embeds is precisely the rank of
$CM'(X)$.
\end{theorem}

\begin{proof}
We outline only one direction, which is the illuminating calculation
relating distances and inner products: if $X$ is Euclidean, so that we
identify each $x_i$ with an isometrically embedded copy in $\R^k$, then
\begin{align}\label{Eschoenberg35}
CM'(X)_{ij} = &\ d_{i0}^2 + d_{j0}^2 - d_{ij}^2 = \| x_i - x_0 \|^2 + \| x_j
- x_0 \|^2 - \| (x_i-x_0) - (x_j-x_0) \|^2 \notag\\
= &\ 2 \tangle{ x_i-x_0, x_j-x_0 }.
\end{align}
But these form a Gram matrix, which is positive semidefinite (see
Theorem~\ref{Tbasics}).
\end{proof}

Parallel to flat space embeddings, Schoenberg also characterized when a
finite metric space embeds into a Euclidean sphere of unit radius,
$S^{r-1} \subset \R^r$. Recall that on any such sphere (for any $r$), any
two antipodal points have distance $\pi$, while two non-antipodal points
$x,y$ and the origin lie on a unique plane, which slices the sphere along
a great circle. Now the intrinsic spherical distance $\sphericalangle
x,y$ equals the length of the shorter arc joining them:
\begin{equation}\label{Esphericaldist}
\sphericalangle x,y := \arccos \tangle{x, y}, \qquad \text{i.e.,} \qquad
\tangle{x,y} = x \cdot y = \cos \sphericalangle x,y, \qquad \forall x,y
\in S^{r-1}.
\end{equation}

This metric exists on each $S^{r-1}$, hence on their union over $r \geq
2$, and hence on its closure -- which is the unit sphere $S^\infty
\subset \ell^2$. This means that applying $\cos(\cdot)$ entrywise to a
spherical distance matrix $(\sphericalangle x_i, x_j)_{i,j=0}^n$ in
$S^\infty$ yields a Gram matrix. The converse is also not hard to show,
leading to another 1935 result of Schoenberg:

\begin{theorem}[{\cite[Theorem~2]{Schoenberg35}}]\label{Tsphere}
A finite metric space $X = \{ x_0, \dots, x_n \}$ embeds isometrically
into a Euclidean unit sphere (with its intrinsic angle metric)
if and only if $X$ has diameter $\leq \pi$ and the entrywise map
$\cos[-]$ sends its distance matrix $D_X$ to a positive semidefinite
matrix $\cos[D_X]$. Moreover, the smallest Euclidean dimension $r$ for
which $X$ embeds spherically in $S^{r-1}$ is the rank of $\cos[D_X]$.
\end{theorem}

\subsection{Positive definite functions}

We make two observations about Theorem~\ref{Tsphere}:
(a)~matrix positivity again plays an indispensable role in it;
and
(b)~the theorem features an early occurrence of the entrywise calculus:
maps that take distance matrices to positive ones; or via composing with
the metric, two-variable symmetric maps that take $X \times X$ to
$\bp_{|X|}(\R)$. This notion was abstracted into

\begin{definition}\hfill\label{Dposdef}
\begin{enumerate}
\item (Schoenberg, 1938, \cite{Schoenberg38}.)
A \textit{positive definite function} on a metric space $(X,d_X)$ is a
map $f : [0,\infty) \to \R$, such that for any points $x_1, \dots, x_n
\in X$, the matrix $(f(d_X(x_i,x_j)))_{i,j=1}^n$ is positive
semidefinite.

\item (Mercer, 1909, \cite{Mercer}; Mathias, 1923, \cite{Mathias}.)
This is the ``traditional'' -- and different -- definition: a
\textit{positive definite function} on a group $G$ is a map $f : G \to
\mathbb{C}$ such that for any points $g_1, \dots, g_n \in G$, the matrix
$(f(g_i^{-1} g_j))_{i,j=1}^n$ is positive semidefinite.
\end{enumerate}
\end{definition}

Indeed, it is the latter notion that appears in the
Carath\'eodory--Toeplitz--Herglotz results discussed after
Theorem~\ref{Trudin} above. We again digress here with a discussion of
the early history of such functions, following the comprehensive
survey~\cite{Stewart} by James Stewart.\footnote{The reader who has
taught freshman calculus, may recognize this author for a different
reason.} Stewart was a student of J.L.B.~Cooper (who descended from
Titchmarsh and hence Hardy), and the survey is an offspring of his
doctoral dissertation on ``Positive definite functions and
generalizations.''

As the name suggests, positive definite functions send ``squares of
domain sets'' to psd matrices.
Indeed, following the use of Smith~\cite{Smith} and others of ``positive
definite'' for matrices, it was Mercer~\cite{Mercer} who extended in~1909
the notion to \textit{kernels}: these are maps $f : X \times X \to
\mathbb{C}$ for an arbitrary set $X$, such that $f(x,x') =
\overline{f(x',x)}$ and the quadratic form induced by $f$ is positive
semidefinite. Examples of such kernels had also appeared in Hilbert's
1904--1910 articles on integral equations~\cite{Hilbert}.

Later, Mathias~\cite{Mathias} independently rediscovered in~1923 such
kernels -- on the group $(\R,+)$. He called them ``positive definite'';
observed using the Schur product theorem that they form a
multiplicatively closed convex cone; and showed that if $f$ is such a
kernel on $(\R,+)$ then its Fourier transform $\widehat{f}(t) := \int_\R
e^{-itx} f(x) dx$ is nonnegative (if it exists). In a sense this is
Bochner's theorem over $\R$, but the connection to a density function on
$\R$ was shown by Bochner a decade later:

\begin{theorem}[\cite{Bochner1}]
A continuous map $f : \R \to \R$ is positive definite if and only if
there exists a (unique) probability measure $\mu$ on $\R$ whose
Fourier--Stieltjes transform is $f$:
\[
f(x) = \int_\R e^{itx} d \mu(t).
\]
\end{theorem}

We end this part with two historical remarks.

\begin{remark}
Bochner also proved his eponymous theorem for $G = \R^r$ in the following
year~\cite{Bochner2}; the general version for locally compact abelian
(LCA) groups came a few years later, due to Povzner, Raikov, and Weil.
Other uses of positive definite kernels included Moore and Aronszajn's
works on reproducing kernels (see e.g.~\cite{PR}), the theory of
harmonics on homogeneous spaces (Cartan, Weyl, It\^o), and its
comprehensive generalization by Krein~\cite{Krein} -- which subsumes the
preceding three authors' work as well as that of Gelfand and Raikov. See
\cite[Section~8]{Stewart} for more details. For a more recent survey
taking the reader from positive definiteness to harmonic analysis and
operator theory, see Shapiro's article~\cite{Shapiro}.
\end{remark}

\begin{remark}
Bochner-type theorems characterize positive definite functions on a LCA
group $G$, as Fourier transforms of probability measures on the dual
group $\widehat{G}$. From this perspective, one should note that the
``general'' LCA-form of Bochner's theorem was first proved by Herglotz
two decades prior to Bochner -- see Theorem~\ref{Therglotz} over the dual
groups $(S^1, \mathbb{Z}$).
\end{remark}

We will mention Herglotz next in Section~\ref{Sdelsarte} in the Appendix
on sphere packings -- given that he is credited by M\"uller with the
Addition Theorem for spherical harmonics.

\subsection{Reconciling the two notions of positive
definiteness}\label{Sposdef}

Having discussed the ``group'' notion of positive definiteness in
Definition~\ref{Dposdef}(2), we now come to Schoenberg's work on positive
definite functions. He introduced the ``metric'' variant in
Definition~\ref{Dposdef}(1) in~\cite{Schoenberg38}, and showed that a
metric space $X$ is Euclidean (or embeds isometrically in $\ell^2$) if
and only if it is separable and the Gaussian family $\{ \exp(-cx^2) : c >
0 \}$ is ``metric'' positive-definite on $X$. Schoenberg then turned his
attention to such positive-definite functions on (Euclidean and
Hilbertian) spheres, in the 1942 paper~\cite{Schoenberg42} bearing this
title.

We explain briefly here, why in the Euclidean sphere context,
Schoenberg's setting reconciles with the ``other'', group-setting that
was introduced by Mercer/Mathias and taken forward by Bochner and others.
More precisely, we explain how each positive definite function on the
sphere, in Schoenberg's metric setting, is a $G$-invariant positive
definite kernel (which incorporates both the ``metric'' and ``group''
settings), and on a two-point homogeneous space $G/H$. We first introduce
the relevant notions.

\begin{definition}\hfill\label{D2pthom}
\begin{enumerate}
\item A compact metric space $(X,d_X)$ is said to be \textit{two-point
homogeneous} if there is a compact group $G$ acting on $X$, such that
given points $p,q,r,s \in X$ with $d_X(p,q) = d_X(r,s)$, there exists $g
\in G$ with $gp=r, gq=s$.

\item We also isolate the common structure underlying both notions in
Definition~\ref{Dposdef}: given a topological space $X$, a continuous
kernel $: X \times X \to \mathbb{C}$ is \textit{positive definite} if
given any $k \geq 1$ and any points $x_1, \dots, x_k \in X$, the matrix
$(K(x_i,x_j))_{i,j=1}^k \in \bp_k(\mathbb{C})$.
\end{enumerate}
\end{definition}

\begin{remark}
Schoenberg's (continuous) positive definite functions are real-valued
kernels which factor through the distance map: $K(\cdot,\cdot) = f \circ
d_X(\cdot,\cdot)$ -- whereas the Mathias--Mercer definition factors
through ``fractions in $G$'': $K(\cdot,\cdot) = f \circ
\eta(\cdot,\cdot)$, where $\eta(g,h) = g^{-1} h$.
\end{remark}

Two-point homogeneity is a strong condition; for instance, setting
$q=p,s=r$ implies that $G$ acts transitively on $X$. Let us fix a
basepoint $x_0 \in X$ and denote $H := {\rm Stab}_G(x_0)$ (a compact
subgroup of $G$), so that $X = G/H$.

\begin{remark}
If $G$ is infinite and connected, 
then $X$ turns out to be a rank-one Riemannian symmetric space of compact
type; these have been classified by Wang~\cite{Wang}. In this case,
$(G,H)$ is also an example of a Gelfand pair~\cite{Gelfand}.
\end{remark}

Let us now turn our attention to a distinguished \textbf{special case}:

\begin{proposition}
Let $r \geq 2$ and let $X = (S^{r-1}, \sphericalangle)$, the unit sphere
in $\R^r$. This is a two-point homogeneous space as above, with $G =
SO(r)$, and $H = SO(r-1) \oplus \{ 1 \}$ for $x_0 = {\bf e}_r$, the north
pole.
\end{proposition}

\begin{proof}
Let $u \in S^{r-1}$ be any unit vector; then $u$ can be completed to an
ordered orthonormal basis $(u_1, \dots, u_{r-1}, u)$ such that the
orthogonal matrix $U = [u_1 | \cdots | u_{r-1} | u]$ has determinant one.
Thus $SO(r)$ acts transitively on $S^{r-1}$, and the stabilizer of $x_0 =
{\bf e}_r$ is $SO(r-1) \oplus 1$. Note that $G$ acts on $X = S^{r-1}$ by
isometries, since $\tangle{\cdot,\cdot}$ is $O(r)$-invariant.

It remains to show two-point homogeneity. Given $a,b,c,d \in S^{r-1}$
with $\sphericalangle a,b = \sphericalangle c,d$, we may first apply the
$G$-transitive action to assume that $a=c={\bf e}_r$. We will use here --
and below -- the basic fact~\eqref{Esphericaldist} that the spherical
distance on $S^{r-1}$ is connected to the inner product. Thus, writing
$b = (b', b_r), d = (d', d_r)$ as tuples in $\R^r$, it follows that $b_r
= d_r$, and hence $\| b' \| = \| d' \|$. From above, there exists $g \in
SO(r-1) \oplus \{ 1 \}$ sending $b' \mapsto d'$ and hence $b \mapsto d$.
Therefore $g \in H$ sends $b,{\bf e}_r$ to $d,{\bf e}_r$, respectively.
\end{proof}

We next show:

\begin{lemma}\label{Lsame}
Let $X = S^{r-1}, G, H, x_0$ be as above, and let $K$ be a continuous
kernel on $X \times X$ into any codomain $C$. Then $K$ is invariant under
the diagonal action of $G$ if and only if 
$K(x,y) \equiv \psi(x \cdot y)\ \forall x,y \in X$
for some continuous map $\psi : [-1,1] \to C$.
\end{lemma}

\begin{proof}
The sufficiency is clear. Conversely, two pairs $(x,y)$ and $(x',y')$ in
$(S^{r-1})^2$ are $G$-conjugate if and only if $\sphericalangle x,y =
\sphericalangle x',y'$, if and only if $x \cdot y = x' \cdot y'$.
\end{proof}

We can now conclude the analysis. A continuous map $f : [0,\pi] \to \R$
is a positive definite function on $(S^{r-1}, \sphericalangle)$ in the
Schoenberg/metric sense, if and only if
\begin{equation}\label{Econvert}
\psi := f \circ \cos^{-1} : [-1,1] \to \R
\end{equation}
sends Gram matrices from $S^{r-1}$ to positive semidefinite matrices. By
Lemma~\ref{Lsame}, this is equivalent to a continuous $G$-invariant
positive definite kernel $K : S^{r-1} \times S^{r-1} \to \R$. \qed

\subsection{Schoenberg: from positive definite functions on spheres to
positivity preservers}

Having reconciled the two notions of positive definiteness (and subsumed
them by kernels in Definition~\ref{D2pthom}), we conclude this section on
Schoenberg's motivations. As we saw above, Schoenberg began his studies
in metric geometry by understanding metric embeddings into Euclidean
spaces and spheres. The latter case naturally led him to introduce
``metric'' positive-definite functions, which he then studied (in the
cited and other papers).

Schoenberg then classified in~1942 all continuous positive definite
functions on $S^{r-1}$ for each $r \geq 2$ (which he wrote composing with
$\cos^{-1}$, so with domain $x \cdot y \in [-1,1]$).

\begin{theorem}[{\cite[Theorem~1]{Schoenberg42}}]\label{Tschoenberg-pd}
Fix an integer $r \geq 2$. Given a continuous map $f : [-1,1] \to \R$,
the following are equivalent:
\begin{enumerate}
\item $f\circ \cos$ is positive definite on $(S^{r-1}, \sphericalangle)$.

\item The map $f(x) = \sum_{k=0}^{\infty} c_k C_k^{(r)}(x)$, where
$c_k\geq 0\ \forall k$, and for varying $r$, the functions $\{
\frac{C_k^{(r)}(x)}{C_k^{(r)}(1)} : k \geq 0 \}$ are precisely the first
Chebyshev, Legendre, or Gegenbauer orthogonal polynomials.
\end{enumerate}
\end{theorem}

Below, we reinterpret Schoenberg's result from a modern perspective;
and further below in the Appendix -- see Section~\ref{Sdelsarte} --
we will derive the ``easier'' half $(2) \implies(1)$ from the theory of
spherical harmonics. Here we continue to the \textit{Hilbert sphere}
$S^\infty$, which is the closure in $\ell^2$ of $\bigcup_{r \geq 2}
S^{r-1}$. Continuous positive definite functions on $S^\infty$ are in a
sense the ``intersection'' of the maps in Theorem~\ref{Tschoenberg-pd}
over all $r$, and Schoenberg showed in the same work:

\begin{theorem}[{\cite[Theorem~2]{Schoenberg42}}]\label{Tsinfty}
A continuous map $f : [-1,1] \to \R$ is positive definite on $S^\infty$
if and only if
\[
f(\cos \theta) = \sum_{k \geq 0} c_k \cos(\theta)^k, \qquad \theta \in
[0,\pi]
\]
where all $c_k \geq 0$ and $\sum_k c_k < \infty$.
\end{theorem}

One can free this result from the sphere context by setting $x = \cos
\theta$. The result then translates into: $f[-]$ sends spherical Gram
matrices to positive semidefinite matrices (see~\eqref{Econvert} and the
next lines) if and only if $f(x) = \sum_{k \geq 0} c_k x^k$ on $[-1,1]$.
(See Section~\ref{Sfixeddim} for a longer explanation.) This is precisely
Schoenberg's theorem~\ref{Tschoenberg} on entrywise positivity
preservers.

\subsection{From Schoenberg's positive definite functions to sphere
packings}

Before we discuss the connections of Schoenberg and Rudin's results to
applied fields, and a modern perspective on positivity preservers, we
mention that Schoenberg's theorem~\ref{Tschoenberg-pd} has another
interesting and important application: to the old problem of sphere
packings in any dimension. As this is not directly related to the
question of preserving positivity, we defer this discussion to the
Appendix; the connection to Schoenberg's result is given in
Section~\ref{Sdelsarte}.

\section{Modern motivations; fixed dimension}

Interest in positivity preservers has also been renewed in recent times,
because of their relevance in modern statistical methodology to analyze
high-dimensional data. In classical statistics, one typically considers a
few random variables ($p$) and has a large number of observations/data
points of them. That is, the sample size $n$ is much larger than $p$,
which renders robustness to traditional statistical estimators. However,
in recent times a new paradigm has emerged, wherein the number of random
variables is very large ($\sim$100,000) and the number of data points is
consequently very small. This motivates strongly, and leads to, many
novel results on entrywise positivity preservers in a fixed dimension, as
well as their connections to hitherto disparate areas in mathematics:
symmetric function theory, finite fields, and graph theory, to name a
few. We refer the reader to the survey~\cite{BGKP-survey2} as well as the
monograph~\cite[Chapter~7]{AK-book} for more on this; here we provide
only a short exposition.

\subsection{An application to machine learning}

Before discussing the fixed dimension setting in earnest, we begin with a
quick digression into an application of entrywise positivity preservers
in the ``dimension-free'' case, i.e., for matrices of all sizes. In this
context, we continue to discuss positive definite kernels on real Hilbert
spaces $\mathcal{H}$: as in Theorem~\ref{Tchrres2}, these are maps $K :
\R \to \R$ such that $(K( \tangle{x_i,x_j}))_{i,j=1}^k$ is positive
semidefinite for all choices of points $x_1, \dots, x_k \in \mathcal{H}$.

It turns out that every such map $K$ yields a reproducing-kernel Hilbert
space, and this is an important tool in Machine Learning (see, e.g.,
\cite{Pinkus04,Steinwart,Vapnik}). By Lemma~\ref{Lbasics}, the Gram
matrices $(\tangle{x_i,x_j})$ are precisely the positive semidefinite
matrices (of all sizes) whose rank is bounded above by $\dim
\mathcal{H}$. Thus, Rudin's theorem~\ref{Trudin} yields the complete
classification of all such kernels -- and in~1959, predating their
application in machine learning by decades.

The next subsection describes how Schoenberg's theorem~\ref{Tschoenberg}
also contains in it the seeds of -- and similarly predates by decades --
the concept of \textit{regularization} in modern high-dimensional data
analysis. For now, we remark for completeness that Schoenberg's results
on positive definite functions on spheres also have applications to
Gaussian random fields, pseudo-differential equations with radial basis
functions, and approximating functions and interpolating data on spheres
-- such as the earth in geospatial modeling. Moreover, Schoenberg's
theorem~\ref{Teuclidean} on Euclidean embeddings of metric spaces is a
crucial ingredient in multidimensional scaling~\cite{CoxCox}.

\subsection{Rudin vs.\ Schoenberg}

We take a moment to compare and contrast the two results:
Rudin's Theorem~\ref{Trudin} vs.\ Schoenberg's
Theorem~\ref{Tschoenberg-pd}.
In both theorems, the test matrices consist of Gram matrices of
arbitrarily large finite sets of vectors drawn from a subset $X \subset
\R^r$:
\begin{itemize}
\item For Rudin, $X$ was the open ball of any fixed radius $0 < \rho \leq
\infty$, and his classification was independent of the precise value of
$r \geq 3$. In this case, the positivity preservers on this test set are
positive sums of monomials.

\item For Schoenberg, $X$ was the unit sphere $S^{r-1}$, which led to the
test matrices being correlation matrices -- so all diagonal entries are
$1$ -- and the rank bounded above by $r$. In this case, the continuous
preservers of positivity on this test set are positive sums of Gegenbauer
polynomials.
\end{itemize}

\subsection{Applied fields and regularization}

We now come to the study of preservers in fixed dimension, strongly
motivated by applied fields. Let us consider a concrete example:
analyzing temperatures on the earth's surface. This has been a subject of
intensive study in recent years in climate science. With the advent of
technology, there are thousands of weather stations and other locations
where one has temperature markers.

Another concrete example involves understanding gene-gene interactions,
in trying to detect the early onset/early warning for cardiovascular
diseases or cancer, say. There are thousands of genes that are studied
today. Alternately, one can consider the behavior over time of financial
instruments (e.g.\ in the stock market) and their interdependencies.

The common theme in all of these examples is of ``complex multivariate
structures'' and trying to understanding their interactions. This is a
key challenge in modern applied fields -- and one of the simplest
measures of dependency between any two such variables is the linear
dependency, captured in their covariance or correlation:
\[
Cov(X,Y) := \mathbb{E}[ (X - \mathbb{E}X) (Y - \mathbb{E}Y) ].
\]
Covariance analysis has been a leading and robust mechanism for data
analysis. The difference now is that the sample covariance matrices
\begin{equation}\label{Esamplecov}
\widehat{\Sigma} = \frac{1}{n-1} \sum_{i=1}^n (x_i - \overline{x})
(x_i - \overline{x})^T
\end{equation}
are enormous in dimension, since there are $p \gg 0$ random variables
being measured. Moreover, as the sample size $n$ is very small, the $p
\times p$ matrix $\widehat{\Sigma}$ is highly singular, which is
unfavorable to subsequent statistical analysis.

Thus, various workarounds have been suggested to ``improve the
properties'' of such matrices and render them more amenable to
statistical techniques. A popular approach has been to apply iterative
methods (``compressed sensing'') -- some names in this regard are Donoho,
Daubechies, Candes, and Tao. While these methods work well for a few
thousand random variables, they are too expensive for matrices of order
$\sim$100,000. As a result, new methodologies are called for.

One alternative in the field of (ultra-)high dimensional covariance
estimation, which has emerged in recent times, is to \textit{regularize}
sample covariance matrices. In other words, one applies a regularizing
function \textit{entrywise} to each covariance or correlation.

\begin{example}
A popular regularizer is \textit{hard-thresholding}. Suppose the true
covariance matrix of the population (or of a probability distribution)
underlying the sampled data, is:
$\Sigma = \begin{pmatrix}
1 & 0.2 & 0 \\
0.2 & 1 & 0.5 \\
0 & 0.5 & 1
\end{pmatrix}$.
Whereas the sample covariance matrix is $\widehat{\Sigma} =
\begin{pmatrix}
0.95 & 0.18 & 0.02 \\
0.18 & 0.96 & 0.47 \\
0.02 & 0.47 & 0.98
\end{pmatrix}$.

It is then natural to threshold small entries (i.e., change them to zero
if their absolute value is below a ``threshold''), with the idea that the
variables are actually independent but the observed value is noise:
$\widetilde{\Sigma} = \begin{pmatrix}
0.95 & 0.18 & \mathbf{0} \\
0.18 & 0.96 & 0.47 \\
\mathbf{0} & 0.47 & 0.98
\end{pmatrix}$.
\end{example}

Such an operation applies directly on the cone, induces sparsity (number
of zero entries), and is scalable because it is entrywise -- no iterative
or algorithmic procedures required. It also displays good consistency and
other properties. (See
e.g.~\cite{bickel_levina,hero_rajaratnam,Hero_Rajaratnam2012,Li_Horvath,Zhang_Horvath}
for some papers that study regularization.)
However, one also needs to ensure that the thresholded (or more
generally, regularized) matrix is itself a proxy for the true covariance
-- and in particular, is positive semidefinite. This is where one has a
theoretical gap, in that it is not clear \textit{for which regularizing
entrywise maps is the transformed matrix guaranteed to be positive
semidefinite}. Thus, while the problem emerges directly and naturally
from applied fields, it brings us back squarely to the question of
classifying entrywise positivity preservers.

Moreover, there is motivation from this perspective to study the question
of preservers in a fixed dimension -- because in a given applied problem,
the number of random variables (aka the dimension of the problem) is
known. So there is no need to study dimension-free preservers; indeed,
only studying these restricts one to using power series with nonnegative
coefficients to regularize covariance matrices, which is unnecessarily
restrictive \textit{and} which does not induce or preserve sparsity. This
strongly motivates trying to classify entrywise positivity preservers --
aka regularizing maps -- in fixed dimension.

\subsection{Positivity preservers in fixed dimension: rank and sparsity
constraints}\label{Sfixeddim}

Resuming the narrative from the previous section, Schoenberg's 1942
classification of positive definite functions can be interpreted today in
terms of regularization. Namely, a continuous function is positive
definite on $S^{r-1}$ for some $r$, if and only if --
via~\eqref{Econvert} -- $f \circ \cos$ sends Gram matrices, of any size /
number of vectors but with all vectors in $S^{r-1}$, to positive
semidefinite matrices. By Theorem~\ref{Tbasics} and Lemma~\ref{Lbasics},
this is equivalent to the entrywise maps preserving positivity, on
matrices of any dimension but with rank at most $r$.

Similarly, if one lets the dimension grow unbounded, $f : [-1,1] \to \R$
is positive definite on $S^\infty$ if and only if $f[-]$ is a positivity
preserver on correlation matrices (i.e., psd matrices with diagonal
entries $1$). Thus, Schoenberg's theorems~\ref{Tschoenberg-pd}
and~\ref{Tsinfty} classified the regularizers of correlation matrices of
any size, but constraining (or not) the \textit{rank}.

The question we are discussing here is different: now the dimension
itself is constrained (and hence, so is the rank). This is not only a
natural theoretical ``next step'' after Schoenberg's and Rudin's results,
it is also motivated from the modern perspective of big data (as
discussed above).

Unfortunately, few results are known in this setting. We mention results
along a few of these fronts. First, the problem in fixed dimension $n=1$
is trivial: the preservers are clearly all functions $: [0,\infty) \to
[0,\infty)$. For $n=2$, the situation is more involved, and was resolved
in~1979 by Vasudeva. We write here his result in a slightly more general
setting; the proof is virtually unchanged.

\begin{theorem}[\cite{vasudeva79}; taken from {\cite[Theorem~6.7]{AK-book}}]
Suppose $0 < \rho \leq \infty$, and the domain $I$ is either $[0,\rho)$
or $(0,\rho)$. Now an entrywise map $f : I \to \R$ preserves positivity
on $\bp_2(I)$ if and only if $f$ is nonnegative, nondecreasing, and
multiplicatively midconvex on $I$ (this last means that $f(\sqrt{xy})^2
\leq f(x) f(y)$ for all $x,y \in I$).

In particular, if we set $I^+ := I \setminus \{ 0 \}$, then: any such
function is continuous on $I^+$, and is either never zero or always zero
on $I^+$.
\end{theorem}

The next case is $n=3$. Remarkably, in this case the problem remains
\textbf{open}. Thus, we do not have a classification of the dimension-$n$
positivity preservers for any $n \geq 3$.

In the absence of a classification for arbitrary functions acting on all
matrices in $\bp_n$, several refinements have been proposed and studied,
yielding ``restricted'' preserver results. We list a few of these in
which the \textit{test matrices} are additionally constrained.
\begin{enumerate}
\item One can impose \textit{rank-constraints}: $f[-]$ sends matrices in
$\bp_n$ with rank at most $l$ to matrices in $\bp_n$ with rank at most
$k$. Rank-constraints are natural in both theory and applications: on the
theoretical side, Gram matrices arising from Euclidean spheres $S^{r-1}$
have rank at most $r$ (in Schoenberg's work~\cite{Schoenberg42}), and
Rudin's theorem~\ref{Trudin} as well as Theorems~\ref{Tbgkp-1sided}
and~\ref{Tbgkp-2sided} also had rank constraints, arising from the
support sets of measures on the circle and line, respectively.
On the applied side, rank constraints arise naturally from the sample
size, which is typically small in applications.
For results along these lines, see the 2017 paper~\cite{GKR-lowrank}.

\item Alternately, one can impose \textit{sparsity constraints}: impose
zero entries in pre-fixed positions. These positions may be determined
from a combinatorial perspective (one forms graphs associated to a zero
pattern) or from domain-specific knowledge in applications, where various
random variables are known to be independent or at least uncorrelated,
e.g., Gaussian graphical models. For such results, see the
work~\cite{GKR-sparse}.

\item One can instead also study preservers of structured matrices such
as Toeplitz or Hankel matrices, as was done in Theorems~\ref{Trudin},
\ref{Tbgkp-1sided}, and \ref{Tbgkp-2sided}.
\end{enumerate}

Independently, one can constrain the set of functions that act on test
matrices. We discuss two families; the first consists of \textit{power
functions}. These are a natural family of entrywise maps to consider, and
they are also used in practice to induce sparsity on matrices by sending
very small / spurious observed correlations to (very close to) zero.

By the Schur product theorem, all integer powers $x^k, k \geq 0$
entrywise preserve positivity in every dimension. We now turn to
non-integer powers -- and hence, only consider matrices with positive
entries. The following result was shown in 1977 by FitzGerald and Horn:

\begin{theorem}[{\cite[Theorem~2.2]{FitzHorn}}]\label{Tfitzhorn}
Given an integer $n \geq 2$, the entrywise power $x^\alpha$, $\alpha \in
\R$ preserves positivity on $\bp_n((0,\infty))$ if and only if $\alpha
\in \mathbb{Z}_{\geq 0} \cup [n-2,\infty)$.
\end{theorem}

This interesting result has seen variants for preservers of positive
matrices with negative entries and variants of power maps; rank
constraints; and sparsity constraints. See e.g.~\cite{GKR-crit-2sided,
GKR-critG, Hiai2009}. In all of these papers, entrywise powers preserving
certain matrix properties are classified, and in all cases the solution
set equals the non-negative integers up to some integer $C$, followed by
all real numbers in $[C,\infty)$. This point of phase transition is
called the \textit{critical exponent} (for that particular matrix
property). In particular, in~\cite{GKR-critG} a new graph invariant is
defined and computed for all chordal graphs.

We conclude this part with a significant strengthening of
Theorem~\ref{Tfitzhorn}. It turns out that there is a multiparameter
family of rank two positive matrices, \textit{each of which} encodes the
entire set of power preservers of Loewner positivity. This was shown by
Jain in 2019:

\begin{theorem}[{\cite[Theorem~1.1]{Jain}}]
Given $n \geq 2$, choose distinct positive reals $x_1, \dots, x_n$. Also
let $\alpha \in \R$. Then $x^\alpha$ entrywise preserves positivity on
the matrix $(1 + x_i x_j)_{i,j=1}^n$ if and only if $\alpha \in
\mathbb{Z}_{\geq 0} \cup [n-2,\infty)$.
\end{theorem}

\subsection{Connection to symmetric polynomials}

In addition to these modern results, there is essentially only one
classical result for entrywise preservers of positivity in a fixed
dimension (without any restrictions on the test matrices or the
functions). This is due to Loewner, who wrote it in a letter to Josephine
Mitchell in~1967 (courtesy: Stanford Library Archives). Later, the result
appeared in the 1969 PhD thesis of his student, Roger
A.~Horn~\cite{Horn}, and was subsequently refined
in~\cite{GKR-lowrank,horndet}. We state an alternate version from
\cite[Theorem~4.2]{BGKP-hankel}:

\begin{theorem}[``Stronger'' Loewner--Horn,
{\cite[Theorem~1.2]{Horn}}]\label{Tloewner}
Fix a dimension $n \geq 3$ and a scalar $0 < \rho \leq \infty$, and let
$I = (0,\rho)$. Let $f : I \to \R$ be such that $f[-]$ preserves
positivity on $\bp_2(I)$ and on all Hankel matrices in $\bp_n(I)$ of rank
at most two.
Then $f \in C^{(n-3)}(I)$, with $f, f', \dots, f^{(n-3)}$ nonnegative on
$I$. Moreover, $f^{(n-3)}$ is convex and nondecreasing on~$I$.
\end{theorem}

Interestingly, this result, combined with Bernstein's theorem on
absolutely monotone functions~\cite{Bernstein}, provides a pathway to
proving the dimension-free Schoenberg--Rudin theorem for preservers of
positive matrices with \textit{positive} entries. See e.g.\
\cite{BGKP-hankel} for this treatment in one and several variables.

The proof of Theorem~\ref{Tloewner}, originally due to Loewner, contains
the seeds of a surprising connection to symmetric function theory. We now
describe this connection, which was discovered only about ten years ago,
and is now better understood.

Recall the century-old observation by P\'olya--Szeg\H{o} (or a special
case thereof) from 1925: if $f(x)$ is a polynomial with nonnegative
coefficients, then $f[-]$ entrywise preserves $\bp_n(\R)$ for all $n \geq
1$. Thus, if one seeks preservers of $\bp_n(\R)$ or of
$\bp_n((-\rho,\rho))$ for fixed $n$, then one should expect more
polynomial preservers: ones which have negative coefficients.
However, apart from Vasudeva's $2 \times 2$ characterization (see above),
\textit{no} example of such a polynomial was known for almost a century.
Certainly, work of Fischer--Stegeman~\cite{Fischer92} had shown that if
such a preserver on $\bp_n((-\rho,\rho))$ is to exist, the $n$ nonzero
coefficients of lowest degrees should be positive, as should those of
highest degrees if $\rho = \infty$. This is ``morally'' along the lines
of the Loewner--Horn theorem~\ref{Tloewner}. The question remains:

\textit{Can at least one other coefficient be negative? More generally,
which coefficients in a polynomial preserver of $\bp_n((-\rho,\rho))$ can
be negative,  for fixed $n \geq 1$ and $\rho \in (0,\infty)$?}

This was first answered in~2016 for a special class of polynomials by
Belton--Guillot--Khare--Putinar~\cite{BGKP-fixeddim}, and then in~2021
for all polynomials by Khare--Tao:

\begin{theorem}[\cite{KT}]\label{TKT}
Fix an integer $n \geq 1$ and a scalar $0 < \rho < \infty$. If a
polynomial $f(x)$ with real coefficients is an entrywise positivity
preserver of $\bp_n((-\rho,\rho))$, then its first $n$ coefficients (of
lowest degree) are positive. Moreover, it is possible for every other
coefficient to be negative.

If $f$ has precisely $n+1$ nonzero coefficients, there exists a
closed-form expression for the negative threshold for the unique negative
coefficient possible (which is the leading term).
\end{theorem}

As the focus of this article is on studying the dimension-free preservers
(originally classified by Schoenberg and then Rudin), we do not provide
more details here, referring the reader to the aforementioned papers, as
well as the survey~\cite{BGKP-survey2} and the monograph~\cite{AK-book}.
We mention, however, that the determinantal calculation that is at the
heart of obtaining the closed-form threshold in Theorem~\ref{TKT}, as
well as at the heart of Loewner's theorem~\ref{Tloewner}, is as follows:

\textit{Given a smooth function $f(t)$, and real scalars $u_i, v_i$ for
$1 \leq i \leq n$, compute the Taylor coefficients of the $n \times n$
determinant} $\Delta(t) := \det f[(t u_i v_j)_{i,j=1}^n]$.

The first $\binom{n}{2}+1$ such derivatives/Taylor coefficients were
worked out by Loewner in the 1960s. However, \textit{all} Taylor
coefficients had been worked out in the special case $f(t) = 1/(1-t)$ by
Cauchy~\cite{Cauchy}, more than a century ago! This is the famous Cauchy
Determinantal Formula, which is an important result in symmetric function
theory, and which turns out to involve \textit{Schur polynomials} in the
$(x_i)$ and in the $(y_j)$ separately. In the 1880s, Frobenius
generalized this to $f$ a sum of two geometric series~\cite{Frobenius}.

Eventually, this question was settled by Khare in full generality -- in
both the analytic and algebraic settings. Here is the latter result (and
it subsumes the calculations by Cauchy, Frobenius, and Loewner):

\begin{theorem}[{\cite[Theorem~2.1]{horndet}}]\label{Tsymm}
Fix a commutative unital ring $R$, and let $t$ be an indeterminate.
Let $f(t) := \sum_{M \geqslant 0} f_M t^M \in R[[t]]$ be an arbitrary
formal power series.
Given vectors $\bu, \bv \in R^N$ for some $N \geqslant 1$, we
have:
\[
\Delta(t) := \det f[ t \bu \bv^T ] = V(\bu) V(\bv) \sum_{M \geqslant
\binom{N}{2}} t^M \sum_{{\bf n} = (n_N, \dots, n_1) \; \vdash M} s_{\bf
n}(\bu) s_{\bf n}(\bv) \, \cdot \prod_{j=1}^N f_{n_j},
\]
where $V(\bu) := \prod_{i<j} (u_i - u_j)$ is the Vandermonde determinant
for $\bu$, and similarly for $V(\bv)$.
\end{theorem}

This result (or its analysis counterpart, wherein the inner sum is
precisely the $M$th Maclaurin coefficient of $\Delta(t)$) yields the
interesting conclusion that \textit{every smooth function / power series
``gives rise to'' all Schur polynomials}. This bridges analysis and
symmetric function theory, and also helps explain the appearance of Schur
polynomials in the preserver problem in fixed dimension
in~\cite{BGKP-fixeddim,KT}.

Later, Khare and Sahi~\cite{Sahi} worked out the analogue of
Theorem~\ref{Tsymm} for the matrix permanent ${\rm perm} f[(t u_i
v_j)_{i,j=1}^n]$, and in fact for all irreducible characters -- even more
generally, all complex class functions -- of the symmetric group and of
every subgroup. Thus, the entrywise calculus also connects (surprisingly)
to group representations and symmetric functions.

As a concluding trivia, we sketch in Figure~\ref{Figtree} the closely
knit academic lineage of several of the experts in this area, having
mentioned some of their contributions above.\footnote{Fej\'er had a
remarkable
\href{https://www.genealogy.math.ndsu.nodak.edu/id.php?id=7488}{list of
PhD students} -- here we name some others who feature in this article and
the Appendix: Paul Erd\"{o}s, L\'aszl\'o Fejes T\'oth, P\'al Tur\'an, and
John von Neumann.}

\begin{figure}[ht]\label{FigMathG}
\begin{tikzpicture}[line cap=round,line join=round,>=triangle 45,x=1.0cm,y=1.0cm]
\draw (-3,6.7) node[anchor=north west] {K.T.W. Weierstrass};
\draw (7.2,6.7) node[anchor=north west] {E.E. Kummer};
\draw (-4,4.1) node[anchor=north west] {H.A. Schwarz};
\draw (0,4.1) node[anchor=north west] {\color{blue}F.G. Frobenius};
\draw (4.6,4.1) node[anchor=north west] {L.I. Fuchs};
\draw (8,4.1) node[anchor=north west] {L. K\"onigsberger};
\draw (-3.6,2) node[anchor=north west] {L. Fej\'er};
\draw (-3.7,0.5) node[anchor=north west] {\color{blue}G. P\'olya};
\draw (2.4,2) node[anchor=north west] {\color{blue}I. Schur};
\draw (1.7,0.5) node[anchor=north west] {\color{blue}I.J. Schoenberg};
\draw (8.75,2) node[anchor=north west] {G. Pick};
\draw (8.4,0.5) node[anchor=north west] {\color{blue}C. Loewner};
\draw (8.5,-1) node[anchor=north west] {\color{blue}R.A. Horn};
\draw (-0.2,-1) node[anchor=north west] {\it (Source: Math-Genealogy)};
\draw (-1.8,6)-- (-2.6,4);
\draw (8.8,6)-- (-2.6,4);
\draw (-1.8,6)-- (1.5,4);
\draw (8.8,6)-- (1.5,4);
\draw (-1.8,6)-- (5.5,4);
\draw (8.8,6)-- (5.5,4);
\draw (-1.8,6)-- (9.6,4);
\draw (8.8,6)-- (9.6,4);
\draw (2.325,3.5)-- (3.5,2.8);
\draw (4.675,3.5)-- (3.5,2.8);
\draw (3.5,1.35)-- (3.5,0.5);
\draw (-2.6,3.5)-- (-2.6,2);
\draw (-2.6,1.35)-- (-2.6,0.5);
\draw (3.5,2.8)-- (3.5,2);
\draw (9.6,3.5)-- (9.6,2);
\draw (9.6,1.35)-- (9.6,0.5);
\draw (9.6,-0)-- (9.6,-1);
\end{tikzpicture}
\caption{Math-Genealogy of some of the experts in positivity, its
preservers, and connections}\label{Figtree}
\end{figure}
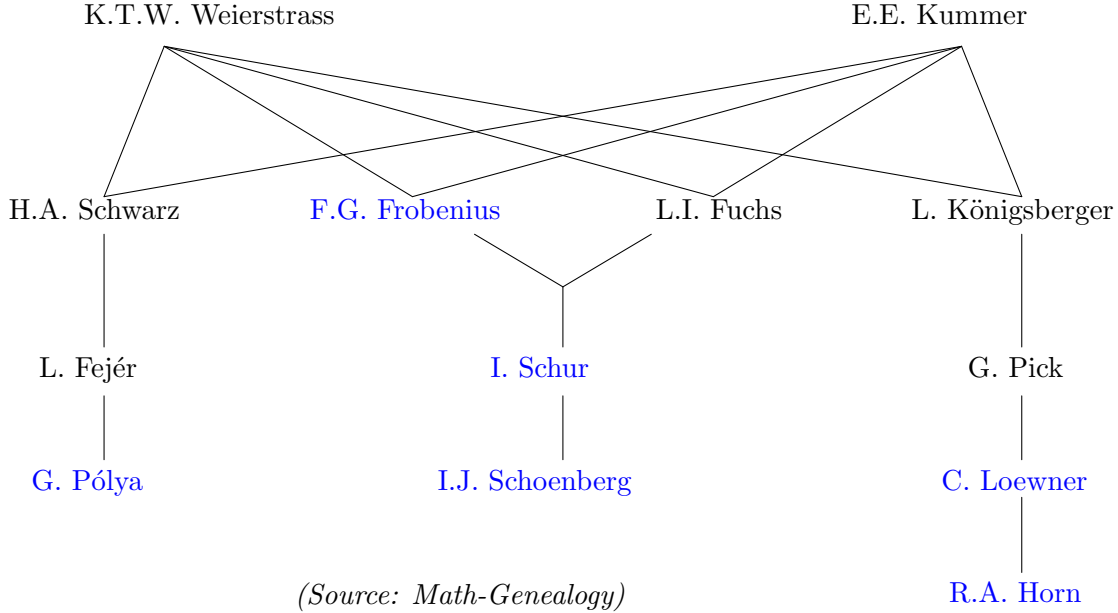

\subsection{Acting only on off-diagonal entries}

A variant of Schoenberg's theorem~\ref{Tschoenberg}, with a modern twist,
is as follows. Recall that Schoenberg was classifying the entrywise maps
sending Gram matrices to themselves -- equivalently, sending covariance
matrices to themselves. Now if the test matrices are correlation matrices
(i.e., Gram matrices of vectors on $S^\infty$), then one may want to
preserve the self-correlations $1$ along the diagonals, while
regularizing the other correlations. Thus, a natural variant of the
entrywise action is as follows:

\begin{definition}
Given a domain $I \subseteq \R$, a function $f : I \to \R$, and a square
matrix $A = (a_{ij})$, define $f^*[A]$ to have diagonal entries $a_{ii}$
and other entries $f(a_{ij})$.
\end{definition}

In 2015, Guillot and Rajaratnam showed that, perhaps surprisingly, the
dimension-free off-diagonal entrywise preservers are once again solutions
to Schoenberg's theorem -- but with an additional constraint:

\begin{theorem}[{\cite[Theorem~4.21]{Guillot_Rajaratnam2012b}}]
Fix a scalar $0 < \rho \leq \infty$ and let $I := (-\rho, \rho)$.
The following are equivalent for a function $f : I \to \R$:
\begin{enumerate}
\item $f^*[-]$ preserves positive semidefiniteness on $\bp_n(I)$ for all
sizes $n \geq 1$.

\item The function $f(x) = \sum_{k \geq 0} c_k x^k$ on $I$ for some
scalars $c_k \geq 0$, and such that $|f(x)| \leq |x|$ on $I$.
(So if $\rho = \infty$, then $f(x) \equiv cx$ on $\R$ with $c \in
[0,1]$.)
\end{enumerate}
\end{theorem}

This line of inquiry was taken forward by Vishwakarma~\cite{Vishwakarma},
who generalized the problem in two ways. First, the function $f$ now
avoids prescribed principal submatrices / diagonal \textit{blocks} in
each dimension, not just $1 \times 1$ blocks / diagonal entries;
and second, on these diagonal blocks a different function $g$ acts.
Vishwakarma classified most of these cases, for $g(x) = \alpha x^k$ with
$\alpha \in (0,\infty)$ and $k \in \mathbb{Z}_{\geq 0}$. Here is a
special case of his main result.

\begin{theorem}[\cite{Vishwakarma}; taken from {\cite[Theorem 20.3]{AK-book}}]
Fix $0 < \rho \leq \infty$, and let $I = (-\rho,\rho)$ and $f,g : I \to
\R$. Now fix for each $n \geq 1$ a collection $T_n$ of subsets of $[n] :=
\{ 1, \dots, n \}$ -- i.e., $T_n \subseteq 2^{[n]}$. Given $A \in I^{n
\times n}$, define $(g,f)_{T_n}[A] \in \R^{n \times n}$ to be the matrix
with $(i,j)$ entry $g(a_{ij})$ if there exists $E \in T_n$ with $i,j \in
E$. If no such $E$ exists then set $(g,f)_{T_n}[A]_{ij} := f(a_{ij})$.

Now assume $T_n \nsubseteq \{ \{ 1 \}, \dots, \{ n \} \}$ for some $n
\geq 3$; and $T_n$ partitions a subset of $[n]$ for all $n \geq 1$. If
$g(x) = \alpha x^k$ with $\alpha \in (0,\infty)$ and $k \in
\mathbb{Z}_{\geq 0}$, and $(g,f)_{T_n}[A] \in \bp_n(\R)$ for all $A \in
\bp_n(I)$, then there are three cases:
\begin{enumerate}
\item If for all $n \geq 3$ we have $T_n = \{ [n] \}$ or $\{ \{ 1 \},
\dots, \{ n \} \}$, then $f$ is a convergent power series with
nonnegative coefficients, and $0 \leq f \leq g$ on $[0,\rho)$.

\item If $T_n$ is not a partition of $[n]$ for some $n \geq 3$, then
$f(x) = c g(x)$ for some $c \in [0,1]$.

\item If neither~(a) nor~(b) holds, then $f(x) = c g(x)$ for some $c \in
[-1/(K-1),1]$, where
\[
K := \max_{n \geq 1} |T_n| \in [2, +\infty].
\]
\end{enumerate}
\end{theorem}

The first case is akin to Schoenberg's theorem, in the form of the final
solution set. The next case is much more restrictive; but it is the third
part of the result that is striking. Recall the profusion of
Schoenberg-type results above, which unanimously reveal absolutely
monotonic functions (power series with nonnegative coefficients) as the
dimension-free preservers. Nevertheless, part~(3) -- in the special case
$g(x) \equiv x$ -- reveals the possibility of $f(x) \equiv cx$ with
$c<0$. This is the first -- and to date, the only -- instance of a
dimension-free setting, in which the positivity preserver is not
absolutely monotonic.

\section{Allowing negative eigenvalues; multivariate
versions}\label{Sinertia}

\subsection{Preservers of matrices with negative inertia}

We now discuss some results that are mostly taken from very recent
preprints, beginning with negative inertia preservers. Having discussed
entrywise maps preserving matrices with all nonnegative eigenvalues, it
is natural to ask what happens if one allows a few negative eigenvalues.
This was worked out by Belton--Guillot--Khare--Putinar
in~\cite{BGKP-inertia}, and we present a few of the findings.

\begin{definition}
Given integers $n \geq 1$ and $0 \leq k \leq n$, and a domain $I
\subseteq \mathbb{C}$, let $\cS_n^{(k)}(I)$ denote the Hermitian $n
\times n$ matrices with all entries in $I$ and exactly $k$ negative
eigenvalues.

Also denote by $\overline{\cS_n^{(k)}}(I)$ the ``closure'', wherein the
entries still stay in $I$ but the \textit{negative inertia} (i.e., number
of eigenvalues $< 0$) is at most $k$: $\displaystyle
\overline{\cS_n^{(k)}}(I) = \bigsqcup_{j=0}^k \cS_n^{(j)}(I)$.

Finally, the \textit{inertia} of an $n \times n$ Hermitian matrix is the
triple $(n_-, n_0, n_+)$, where the coordinates denote the numbers of
negative, zero, and positive eigenvalues, respectively.
\end{definition}

Our goal is to understand the entrywise preservers of negative inertia on
$\bigcup_{n \geq k} \cS_n^{(k)}(I)$ for all integers $k \geq 0$. Note
that the case of $k=0$ is precisely the (dimension-free)
Schoenberg--Rudin theorem above. We state the next result only for $I =
(-\rho,\rho)$ (and in it, compute the inertia preservers as well). The
cases of $I = (0,\rho)$ and $[0,\rho)$ will be presented through their
multivariate versions below.

\begin{theorem}[{\cite[Theorems~1.2 and~1.3]{BGKP-inertia}}]
Fix an integer $k \geq 0$ and a scalar $0 < \rho \leq \infty$. Let $I =
(-\rho,\rho)$ and $f : I \to \R$.
\begin{enumerate}
\item Then $f[-]$ preserves the inertia of all matrices in $\cS_n^{(k)}(I)$
for all $n \geq k$ if and only if $f(x) \equiv cx$ for some $c>0$.

\item The map $f[-]$ preserves the negative inertia of all matrices in
$\cS_n^{(k)}(I)$ for all $n \geq k$ if and only if:
\begin{enumerate}
\item $k=0$: This is the Schoenberg--Rudin theorem, and $f$ must be a
convergent power series on $I$ with nonnegative Maclaurin coefficients.
\item $k=1$: $f(x) \equiv cx$ or $f(x) \equiv -c$, for some $c>0$.
\item $k \geq 2$: $f(x) \equiv cx$ for some $c>0$.
\end{enumerate}
\end{enumerate}
\end{theorem}

Thus, the class of dimension-free (negative) inertia preservers is very
rigid for $k>0$. It turns out that a more interesting question is to
classify the entrywise maps sending $\cS_n^{(k)}(I)$ to
$\overline{\cS_n^{(l)}}$. This too admits a complete solution, whose
multivariate version is perhaps more clarifying to state. For now, we
provide the meat of the assertion in the form of a summary:

\begin{theorem}[{\cite[Theorem A]{BGKP-inertia}}]\label{Tinertia}
Fix nonnegative integers $k,l$ and a scalar $0 < \rho \leq \infty$. Let
$I = (-\rho,\rho)$ and $f : I \to \R$ be such that $f[-] : \cS_n^{(k)}(I)
\to \overline{\cS_n^{(l)}}(\R)$ for all $n \geq k,l$.
\begin{enumerate}
\item If $k=0$, then $f(x)$ equals a(ny) real number $f(0)$ plus a
convergent power series with nonnegative Maclaurin coefficients and
vanishing at $x=0$.

\item If $k>0$, then $f$ is linear or constant.
\end{enumerate}
\end{theorem}

\subsection{Preservers of positivity and of inertia, in several
variables}

Schoenberg's theorem has a natural multivariable generalization, for any
number of variables $m \geq 1$. Note by the Schur product theorem that if
$A_1, \dots, A_m$ are any matrices in $\bp_n(\mathbb{C})$ for some $n
\geq 1$, then their entrywise product $A_1 \circ \cdots \circ A_m \in
\bp_n(\mathbb{C})$. Reformulated, this says that the function $f(\bx) =
x_1 \cdots x_m$ entrywise sends $\bp_n(\mathbb{C})^m$ to
$\bp_n(\mathbb{C})$ for all $n \geq 1$. In general, we define
\[
f[A_1, \dots, A_m]_{ij} := f(a_{ij}^{(1)}, \dots, a_{ij}^{(m)}), \qquad
A_p = (a_{ij}^{(p)})_{i,j=1}^n.
\]

Restricting to real matrices, the easy implication of the following 1995
generalization of Schoenberg's theorem -- by
FitzGerald--Micchelli--Pinkus -- is clear:

\begin{theorem}[{\cite[Theorem~2.1]{fitzgerald}}]
Let $I = \R$. An entrywise map $f : I^m \to \R$ sends $\bp_n(I)^m$ to
$\bp_n(\R)$ for all $n \geq 1$, if and only if $f$ equals a convergent
power series with nonnegative coefficients on $I^m$:
\begin{equation}\label{Emultientire}
f(\bx) = \sum_{\alpha \in \mathbb{Z}_{\geq 0}^m} c_\alpha \bx^\alpha
\qquad \text{for all } \bx \in I^m, \quad \text{with all } c_\alpha \geq
0.
\end{equation}
\end{theorem}

This result was recently strengthened by Belton--Guillot--Khare--Putinar
in two ways. First, the domain was reduced and also varied, to one- and
two- sided domains. Second, the test set of matrices was severely
reduced, to low rank Hankel ones.

\begin{theorem}[\cite{BGKP-hankel}]\label{Tjems}
Let $0 < \rho \leq \infty$, and let $I = (0,\rho), [0,\rho)$, or
$(-\rho,\rho)$. Fix a positive integer $m$ and suppose $f : I^m \to \R$.
The following are equivalent.
\begin{enumerate}
\item The entrywise map $f[-]$ sends $m$-tuples in $\bp_n(I)^m$ to
$\bp_n(\R)$.

\item Let $H_n(I)$ denote the Hankel matrices in $\bp_n(I)$ -- of rank at
most $2$ for $I = (0,\rho), [0,\rho)$, and rank at most $3$ for $I =
(-\rho,\rho)$. Then $f[-] : H_n(I)^m \to \bp_n(\R)$ for all $n \geq 1$.

\item The map $f$ is as in~\eqref{Emultientire}, on all of $I^m$.
\end{enumerate}
\end{theorem}

\begin{proof}[Skeleton of proof]
That $(3) \implies (1)$ is via the Schur product theorem, and that $(1)
\implies (2)$ is immediate. That $(2) \implies (3)$ was shown in
\cite{BGKP-hankel} in the following locations:
\begin{itemize}
\item For $I = (0,\rho)$: see Theorem~{9.6} in \textit{loc.\ cit.}

\item For $I = [0,\rho)$: see Theorem~9.6 and the proof of
Proposition~9.8.

\item For $I = (-\rho,\rho)$: see Theorem~9.11 and the subsequent
remarks. \qedhere
\end{itemize}
\end{proof}

Having understood positivity preservers, we turn to the multivariate
analogue of Theorem~\ref{Tinertia}. The following is
\cite[Theorem~B]{BGKP-inertia}:

\begin{theorem}[Schoenberg-type theorem with negativity
constraints]\label{Tmain}
Let $I := ( -\rho, \rho )$, $( 0, \rho )$, or $[ 0, \rho )$, where $0 <
\rho \leq \infty$. Also fix integers $m \geq 1$ and $k_1, \dots, k_m, l
\geq 0$. Rearrange the negative inertias $k_p$ such that any zero values
are at the start; thus there exists $0 \leq m_0 \leq m$ such that $k_1 =
\cdots = k_{m_0} = 0 < k_{m_0+1}, \dots, k_m$.

Now given any function $f : I^m \to \R$, the following are equivalent.
\begin{enumerate}
\item The entrywise map $f[ - ]$ sends $\times_{p=1}^m
\overline{\cS_n^{(k_p)}}(I)$ to $\overline{\cS_n^{(l)}}(\R)$ for all $n
\geq \max_p k_p$.

\item The entrywise map $f[ - ]$ sends $\times_{p=1}^m
\cS_n^{(k_p)}(I)$ to $\overline{\cS_n^{(l)}}(\R)$ for all $n \geq \max_p
k_p$.

\item There exists a function $F : (-\rho,\rho)^{m_0} \to \R$
and a non-negative constant $c_p$ for each $p = m_0 + 1, \dots, m$
such that
\begin{enumerate}
\item we have the representation
\begin{equation}\label{Epresrep2}
f( \bx ) = F( x_1, \ldots, x_{m_0} ) + \sum_{p = m_0 + 1}^m c_p x_p
\qquad \textrm{for all } \bx \in I^m,
\end{equation}
\item the function
$\bx' := ( x_1, \ldots, x_{m_0} ) \mapsto F( \bx' ) - F( {\bf 0}_{m_0} )$
is absolutely monotone, that is, it is represented on $I^{m_0}$
by a convergent power series with all Maclaurin coefficients
non-negative, and
\item the inequality
$\displaystyle {\bf 1}_{F( {\bf 0} ) < 0} + \sum_{p : c_p > 0} k_p \leq
l$ holds.
\end{enumerate}
\end{enumerate}
\end{theorem}

In addition to judiciously chosen test matrices and analysis techniques,
an interesting additional ingredient in the proof involves the use of
Sidon sets (also termed $B$-sets) from number theory and additive
combinatorics, whose use was pioneered by Singer~\cite{Singer} and
Erd\"os--Tur\'an~\cite{ET}; see also Bose--Chowla~\cite{BC}.

Having discussed the real case, we end by mentioning the multivariable
complex case too. In this case -- following Herz's theorem~\ref{Therz},
its multivariate counterpart was again shown in 1995 by
FitzGerald--Micchelli--Pinkus:

\begin{theorem}[{\cite[Theorem~3.1]{fitzgerald}}]\label{Tfmp}
Given an integer $m \geq 1$ and a function $f : \mathbb{C}^m \to
\mathbb{C}$, the following are equivalent.
\begin{enumerate}
\item The entrywise map $f[-]$ sends $\bp_n(\mathbb{C})^m$ to
$\bp_n(\mathbb{C})$ for all $n \geq 1$.

\item The function $f(\bz) = \sum_{\alpha,\beta \in \mathbb{Z}_{\geq
0}^m} c_{\alpha,\beta} \bz^\alpha \overline{\bz}^\beta$ for all
$\bz \in \mathbb{C}^m$, with all $c_{\alpha,\beta} \geq 0$.
\end{enumerate}
\end{theorem}

This result was extended to classify the preservers of negative inertia.
Here we present the complex analogue of Theorem~\ref{Tmain}; this is
\cite[Theorem~C]{BGKP-inertia}:

\begin{theorem}\label{Tmain-complex}
Let the integers $m, k_1, \dots, k_m, l$ be as in Theorem~\ref{Tmain}.
Given any function $f : \mathbb{C}^m \to \mathbb{C}$, the following are
equivalent.
\begin{enumerate}
\item The entrywise transform $f[ - ]$ sends
$\times_{p=1}^m \overline{\cS_n^{(k_p)}}(\mathbb{C})$ to
$\overline{\cS_n^{(l)}}(\mathbb{C})$ for all $n \geq \max_p k_p$.

\item The entrywise transform $f[ - ]$ sends $\times_{p=1}^m
\cS_n^{(k_p)}(\mathbb{C})$ to $\overline{\cS_n^{(l)}}(\mathbb{C})$ for
all $n \geq \max_p k_p$.

\item There exists a function $F : \mathbb{C}^{m_0} \to \mathbb{C}$ and
non-negative constants $c_p$ and $d_p$ for each $p = m_0 + 1, \dots, m$
such that
\begin{enumerate}
\item we have the representation
\begin{equation}\label{Epresrep3}
f( \bz ) = F( z_1, \ldots, z_{m_0} ) +
\sum_{p = m_0 + 1}^m \bigl( c_p z_p + d_p \overline{z_p} \bigr)
\qquad \textrm{for all } \bz \in \mathbb{C}^m,
\end{equation}
\item the function $\bz' := (z_1, \dots, z_{m_0}) \mapsto F( \bz' ) - F(
{\bf 0}_{m_0} )$ is represented on $\mathbb{C}^{m_0}$ by a convergent
power series in $\bz'$ and $\overline{\bz'}$ with non-negative
coefficients, as in Theorem~\ref{Tfmp}, and
\item $f( {\bf 0}_m ) = F( {\bf 0}_{m_0} )$ is real, and we have
$\displaystyle {\bf 1}_{F( {\bf 0} ) < 0} + \sum_{p : c_p > 0} k_p +
\sum_{p : d_p > 0} k_p \leq l$.
\end{enumerate}
\end{enumerate}
\end{theorem}

Thus, the rich class of preservers in this result and Theorem~\ref{Tmain}
mix the absolutely monotone class of Schoenberg, with the rigid class of
nonnegative homotheties. Notice that if all $k_p = l = 0$, we recover
Schoenberg's (multivariate) real and complex theorems.

\section{Recent developments on positivity preservers}

We are nearing the end of this survey. This rather short section has
three relatively disconnected parts.

\subsection{Preservers of positivity and of non-positivity}

In all of the results mentioned above, we have focused on classifying the
functions $f$ such that $f[A]$ is positive if $A$ is so. (We omit the
inertia and multivariate considerations of Section~\ref{Sinertia}.) Here
we consider the natural parallel question: what are the functions such
that $f[A]$ is positive semidefinite \textit{if and only if} $A$ is so?
A close variant is to replace ``positive semidefinite'' by ``positive
definite''.

Here are the answers, from recent work by
Guillot--Gupta--Vishwakarma--Yip. Remarkably, the answers are the same
for all dimensions and for any fixed dimension greater than two:

\begin{theorem}[{\cite[Theorem~1.8]{GGVY-sign}}]\label{Tsign}
Fix a dimension $n \geq 3$ and let $\F = \R$ or $\mathbb{C}$. The
following are equivalent for an arbitrary function $f : \F \to \F$:
\begin{enumerate}
\item $A \in \F^{n \times n}$ is a positive definite matrix if and only
if $f[A]$ is.

\item $A \in \F^{n \times n}$ is a positive semidefinite matrix if and
only if $f[A]$ is.

\item $f$ is a positive multiple of a continuous field automorphism of
$\F$. That is, $f(x) \equiv cx$ if $\F = \R$, and $f(z) \equiv cz$ or $c
\overline{z}$ if $\F = \mathbb{C}$, for some $c>0$.
\end{enumerate}
\end{theorem}

In fact the authors prove significantly stronger results
in~\cite{GGVY-sign}, in that they study the problem for matrices with
entries in a wide class of sub-domains of $\R$ and $\mathbb{C}$.

\subsection{Preservers over finite fields}

We next study Schoenberg's theorem over a nonstandard setting: finite
fields $\F = \F_q$. In analogy to the real case, here one defines a
scalar to be \textit{positive} if it is the square of a nonzero element
-- i.e., a nonzero quadratic residue. Notice that these elements still
form half of the units $\F_q^\times$ when the prime power $q$ is odd.

Defining positive matrices is more challenging. It turns out that over
finite fields, even the basic characterizations of positive semidefinite
matrices -- in Theorem~\ref{Tbasics} -- are not all available. Instead,
we use a different characterization of positive \textit{definite}
matrices, one that is unavailable for semidefinite matrices:

\textit{A real symmetric matrix is positive definite if and only if its
leading principal minors are all positive.}

It turns out that this notion can be adapted usefully to finite fields.
This was studied in detail by Cooper--Hanna--Whitlatch in~\cite{CHW}, and
they showed that when $q$ is even or $q \equiv 3 \mod 4$, such matrices
$A_{n \times n}$ admit a Cholesky decomposition: $A = L L^T$, where $L$
is lower triangular with entries in $\F_q$ and positive diagonal entries.
Thus, we work with:

\begin{definition}
A symmetric $n \times n$ matrix over a finite field $\F_q$ is
\textit{positive definite} if its leading principal $1 \times 1, \dots, n
\times n$ minors are squares of nonzero elements in $\F_q$.
\end{definition}

We now present a natural class of entrywise preservers of positive
\textit{definiteness} over $\F_q$, in the spirit of P\'olya and
Szeg\H{o}'s century-old result. Namely, if ${\rm char}(\F_q) =: p$ and $x
\mapsto x^p$ is the Frobenius automorphism, then
\[
\det A^{\circ p} = \det (a_{ij}^p) = \sum_{\sigma \in S_n} \det(\sigma)
\prod_{i=1}^n a_{i \sigma(i)}^p = \left( \sum_{\sigma \in S_n}
\det(\sigma) \prod_{i=1}^n a_{i \sigma(i)} \right)^p = (\det A)^p.
\]

From this it follows that the same functions as in Theorem~\ref{Tsign}(3)
are preservers: positive multiples of field automorphisms $x \mapsto c
x^{p^k}$, with $c \in \F_q$ positive and $k \geq 0$ in $\mathbb{Z}$.
Remarkably, Guillot--Gupta--Vishwakarma--Yip showed that -- akin to
Theorem~\ref{Tsign} -- for every fixed $n \geq 3$, there are no other
preservers:

\begin{theorem}[\cite{GGVY-finite}]\label{Tfinite}
Let $q = p^\ell$ for a prime $p \geq 2$ and an integer $\ell \geq 1$, and
$f : \F_q \to \F_q$. The following are equivalent.
\begin{enumerate}
\item The map $f[-]$ preserves $\bp_n(\F_q)$ for some $n \geq 3$.
\item The map $f[-]$ preserves $\bp_n(\F_q)$ for all $n \geq 3$.
\item $f(x) = c x^{p^k}$ for some $c \in \F_q$ positive and $0 \leq k
\leq \ell-1$.
\end{enumerate}

If moreover $p$ is odd, these are also equivalent to:
\begin{enumerate}
\setcounter{enumi}{3}
\item $f(0) = 0$ and $f$ is an automorphism of the Paley graph over
$\F_q$, i.e.,
\[
(f(a)-f(b))^{(q-1)/2} = (a-b)^{(q-1)/2} \qquad \forall a,b \in \F_q.
\]
\end{enumerate}
\end{theorem}

Thus, unlike the ``classical'' Schoenberg--Rudin theorem, both the
``if-and-only-if'' versions in Theorem~\ref{Tsign} and
Theorem~\ref{Tfinite} over finite fields admit solutions in each fixed
dimension (and hence in the dimension-free setting).
Moreover, the proof over finite fields is completely different than over
the reals or complex numbers, and involves an interesting mix of tools:
the quadratic character $x^{(q-1)/2}$, the celebrated Weil character
bounds, results of Carlitz on Paley graphs, and the Erd\"os--Ko--Rado
theorem, to name a few.

\begin{remark}
We add that the papers~\cite{GGVY-sign} and~\cite{GGVY-finite} prove many
more results on ``if-and-only-if'' positivity preservers as well as
preservers over finite fields; what is provided above and here are merely
a sample. We refer the reader to these works for more details.
\end{remark}

\subsection{Schoenberg's theorem for integer matrices}\label{SDP}

We end this section by returning full circle to Schoenberg's theorem in
characteristic zero -- but this time over matrices with \textit{integer}
entries. In this case, one works with functions $f : \mathbb{Z} \to \R$
-- which have a discrete domain, so one cannot even take limits of
function-values!

Remarkably, Schoenberg's theorem still holds here, with somewhat stronger
hypotheses:

\begin{definition}
A real matrix is said to be \textit{partially defined} if a subset of its
entries are unspecified. Such a matrix is said to be \textit{positive
(semi)definite} if there exists a choice of each unspecified entry such
that the resulting \textit{matrix completion} is positive (semi)definite.
\end{definition}

We now come to the main new tool that is required to state the result.

\begin{definition}
Given a subset $I \subseteq \R$, a function $f : I \to \R$ is a
\textit{partially defined entrywise positivity preserver} if for every
partially defined psd matrix $A$ with specified entries in $I$, the
partially defined matrix $f[A]$ is also psd.
\end{definition}

In other words, for every partially defined psd matrix $A$ with specified
entries in $I$, both $A$ and $f[A]$ can be completed (via some choices of
real scalars) to psd matrices.

With these notions at hand, we state ``Schoenberg's theorem over
$\mathbb{Z}$'', shown recently by Damase--Pascoe:

\begin{theorem}[\cite{DP}]\label{TDP}
A function $f : \mathbb{Z} \to \R$ is a partially defined entrywise
positivity preserver, if and only if $f(n) = \sum_{k \geq 0} c_k n^k$ for
all integers $n$, with all $c_k \in [0, \infty)$.
\end{theorem}

In other words, as ``usual'' $f$ is the restriction to $\mathbb{Z}$ of an
entire function that is absolutely monotonic on $(0,\infty)$. Note that
one implication is precisely the P\'olya--Szeg\H{o} 1925 observation, via
the Schur product theorem.

\begin{remark}
Theorem~\ref{TDP} is valid even if one replaces $\mathbb{Z}$ by any $X
\subseteq \R$ with $\sup X = \infty$.
\end{remark}

While Theorem~\ref{TDP} is indeed a pleasing addition, it is natural to
ask if one can get rid of the ``partially defined'' test matrices:

\begin{question}
Suppose a (continuous) entrywise map $f : \mathbb{Z} \to \R$ preserves
positivity on $\bigcup_{n \geq 1} \bp_n(\mathbb{Z})$. Is $f$ the
restriction to $\mathbb{Z}$ of an entire function with nonnegative
Maclaurin coefficients?
\end{question}

Note that there can exist extensions of $f$ to the real line that are not
of the claimed form. For instance, the Dirichlet function $f(x) = {\bf
1}_{x \in \mathbb{Q}}$ preserves positivity on integer matrices. Even if
$f$ is required to be entire, the power series $f(x) = \sin (\pi x)$,
which sends every integer matrix to the zero matrix, is a preserver with
negative coefficients. The question above is if there exists any
extension with all Maclaurin coefficients nonnegative. In the absence of
an answer to this question, Damase--Pascoe work with partially defined
entrywise positivity preservers as a necessary sophistication, in order
to obtain a Schoenberg-type result.

\appendix
\section{Sphere packings and kissing numbers in Euclidean
space}\label{Asphere}

This Appendix contains a parallel mini-survey of two famous problems in
discrete geometry: understanding sphere packings and kissing numbers for
spheres in Euclidean space. This also connects to the main body of this
article via Schoenberg's theorem~\ref{Tschoenberg-pd} (which classifies
the dimension-free entrywise positivity preservers with a rank
constraint). See Section~\ref{Sdelsarte}.

Much of this material can be found in the well known
monograph~\cite{CS99} by Conway and Sloane. See also \cite{Zong}, as well
as the classic 1964~text by Rogers~\cite{Rogers}.

The question of sphere packings is informally stated as follows: given a
dimension $n \geq 1$, how does one stack the maximum amount of congruent
balls -- or \textit{spheres} -- $B_{\R^n}(a_k, r)$ for a fixed $r>0$ and
centers $a_1, a_2, \ldots \in \R^n$, which can intersect only along their
boundaries.

\subsection{Problem statement and early history}

We now write the question more formally in order to set notation. We will
work with unit spheres -- i.e., $r=1$ -- and denote each sphere by its
center, so that the condition of non-overlaps is equivalent to all
centers being at least distance $2$ apart.

\begin{definition}
Fix a dimension $n \geq 1$.
\begin{enumerate}
\item A \textit{(sphere) packing} in $\R^n$ (or in general, in any metric
space) is a countably infinite subset (of ``centers'') $\mathcal{P}
\subset (\R^n, \| \cdot \|_2)$ such that $\| x-y \|_2 \geq 2$ for all $x
\neq y$ in $\mathcal{P}$.

\item A packing is a \textit{lattice packing} if $\mathcal{P}$ is a
lattice in $\R^n$.

\item The \textit{density} of a packing $\mathcal{P}$ is (informally) the
maximum proportion of space filled by the spheres, and is defined via the
formula
\begin{equation}
\Delta_{\mathcal{P}} := \lim \sup_{M \to \infty} \frac{1}{(2M)^n}
{\rm Vol}_{\R^n} \left( [-M,M]^n \cap \bigcup_{x \in \mathcal{P}}
B_{\R^n}(x,1) \right).
\end{equation}
\end{enumerate}
\end{definition}

In other words, take a large cube $[-M,M]^n$ in $\R^n$ and intersect it
with $\mathcal{P}$; now compute the percentage of this intersection
inside the cube; and take $M \to \infty$. Finally, we have:

\begin{definition}
The \textit{sphere packing density} of $\R^n$ is $\Delta_{\R^n} :=
\displaystyle \sup_{\text{packings }\mathcal{P}} \Delta_{\mathcal{P}}$.

\noindent On a related note, define the \textit{lattice packing density}
of $\R^n$ to be
\begin{equation}
\Delta^{(L)}_{\R^n} := \sup_{\text{lattice packings }\mathcal{P}}
\Delta_{\mathcal{P}}.
\end{equation}
\end{definition}

The \textbf{question} of interest is to compute $\Delta_{\R^n},
\Delta^{(L)}_{\R^n} \in [0,1]$ for $n>1$ (for $n=1$, they are clearly
$1$). This question has a storied history. Perhaps the earliest version
in modern times involves understanding close-packing of spheres, which
came up in the 1580s when Sir Walter Raleigh asked the English astronomer
and mathematician Thomas Harriot\footnote{On a historical note, Thomas
Harriot made remarkable contributions to mathematics and physics: he was
the first to observe sunspots using a telescope, preceded Galileo in
drawing a map of the moon using a telescope, and is said to have studied
refraction and discovered Snell's law before Snell. In mathematics, he
pioneered the modern way of computing with algebraic unknowns, and proved
Girard's theorem  on the area of a triangle on the unit sphere (predating
Girard).}
about efficiently stacking cannonballs (on ships). Harriot published in
the 1590s an analysis of stacking patterns -- which also led towards
atomic theory -- and later corresponded with his German contemporary
Johannes Kepler, who stated a conjectural form in~{1611} in the
work~\cite{Kepler}: the maximum packing density in three dimensions is
achieved by a pyramidal piling, akin to oranges in grocery stores.

It took almost four centuries to fully (and positively) settle Kepler's
conjecture.\footnote{On a light note: this means in particular that it
isn't only Fermat's Last Theorem that had to wait for centuries for a
resolution.} During this time, many researchers studied the problem in
three \textit{and two} dimensions. For $n=2$: in~1773,
Lagrange~\cite{Lagrange} studied extremal quadratic forms and deduced
that the lattice density in $\R^2$ is achieved by the hexagonal/honeycomb
packing: $\displaystyle \Delta^{(L)}_{\R^2} = \frac{\pi}{2\sqrt{3}}$.
However, the bound for all packings was realized much later. In 1910,
Thue showed~\cite{Thue} that this is indeed the unconstrained packing
density: $\displaystyle \Delta_{\R^2} = \frac{\pi}{2\sqrt{3}}$; however,
it is generally believed that his proof has a gap. A complete proof was
provided in the 1940s by Fejes T\'oth~\cite{LFT40,LFT42}.

The $n=3$ case was harder, and an early result is by Gauss, who showed in
his book review~\cite{Gauss} of Seeber's book on quadratic
forms~\cite{Seeber}\footnote{Seeber was a German mathematician and
physicist who is especially known for his mathematical work focusing on
crystallography. Thus, between Lagrange, Gauss, Seeber, and the French
physicist Auguste Bravais -- who is known for working on the lattice
theory of crystals -- one can regard this as the time when lattices
became formalized and mainstream in mathematics.}
that Lagrange's 1773 methods could be modified to yield the 3-dimensional
analogue:
$\displaystyle \Delta^{(L)}_{\R^3} = \frac{\pi}{3\sqrt{2}}$.
In other words, Kepler's cannonball packing achieves the lattice packing
density in $\R^3$.
(See Section~\ref{Shermite} for the connection of Lagrange's and Gauss'
study of extremal quadratic forms to sphere packings.)

Note that there are not one, but at least two different ways to achieve
the optimal cannonball packing density. Namely, on top of each ``sheet''
of spheres, one arranges the next sheet according to either the
\textit{FCC (face-centered cubic)} arrangement -- also denoted
\textit{CCP} -- or the \textit{HCP (hexagonal close-packed)} arrangement.
As one gets two such choices each time, both with the same density, there
are uncountably many arrangements of ``sheets'' in 3-space which yield
the Kepler packing density. This uncountable family is collectively
termed as \textit{Barlow packings}, in honor of the crystallographer
William Barlow's 1883 work~\cite{Barlow}.

Subsequently, the Kepler conjecture featured in Hilbert's $18^{\rm th}$
problem~\cite{Hilbert-ICM}, at the turn of the 20th century.

\subsection{Applications}

A quick digression: the problem of sphere packings was not only of
intrinsic mathematical interest (or for efficiently stacking cannonballs
or oranges, in real life), it features in other fields too -- starting
with Harriot and Barlow, whose aforementioned works relate sphere
packings to crystallography. Sphere packings are a reasonable starting
point for modeling the structure of gases, liquids, crystals, and
granular media, and can yield density estimates for ``idealized
materials''.

In the mathematical sciences, we now know that sphere packings are
related to number theory (via modular forms) and to optimization. One can
also connect the question to physics, including statistical mechanics and
the Thomson problem. E.g. one imagines the centers of the spheres to be
electrons that repel one another, and the densest packing is a ``minimum
energy configuration'' of electrons.

However, arguably the most important application of (higher-dimensional)
sphere packings is to communication theory: they are the continuous
analogues, in a sense, of length $n$ error-correcting codes. This emerges
from work of Shannon, Hamming, and others. Roughly speaking, one is
transmitting a set of signals -- represented here by points $x \in \R^n$,
with $n$ usually large. (E.g., the coordinates may be amplitudes at
different frequencies, typically in the hundreds or more.) The
transmission channel may have a noise level $\varepsilon > 0$, so one
expects the transmitted signal $x$ to be received at the other end as
some $y \in B_{\R^n}(x,\varepsilon)$. Thus, if one builds the signal set
/ ``vocabulary'' such that any two signals have distance at least $2
\varepsilon$, then this allows the mechanism to ``error-correct''.
Whereas if this is not the case, certain received signals could not be
``decoded'' to recover uniquely the transmitted signal.

Moreover, we would like to have as large a vocabulary as possible --
which translates precisely into maximally packing $\varepsilon$-spheres
in a fixed space. This is an important real-world application:
error-correcting codes are used by cell phones, the internet, and even
space probes to send signals reliably.

\begin{example}\label{ExHamming}
We provide an early example of such a code, and it is discrete in nature.
Given a set $F$ and integers $n, k \geq 1$, we first set some notation.
\begin{enumerate}
\item A \textit{code} (or \textit{$F$-code}) \textit{of length $n$} is a
subset $S \subset F^n$.

\item Such a subset is a \textit{$k$-error correcting code} if the
\textit{Hamming distance} between any $x \neq y \in S$ (i.e., the number
of coordinates where $x_i \neq y_i$) is at least $2k+1$.

\item A $k$-error correcting $F$-code of length $n$ is \textit{perfect}
if every word in $F^n$ is ``detectable'' by $S$ up to error at most $k$
(in particular, the balls in the next equality are disjoint):
\[
F^n = \bigsqcup_{s \in S} B_{Hamming}(s,k).
\]
\end{enumerate}

For instance, here is a binary 1-error correcting code of length 7:
\[
(1, 1, 0, 1, 0, 0, 0), \quad (0, 1, 1, 0, 1, 0, 0), \quad
(0, 0, 1, 1, 0, 1, 0), \quad (0, 0, 0, 1, 1, 0, 1).
\]
It is easy to check that the Hamming distance between any two unequal
points is $4$. Thus, the code can detect (and correct) errors in data
transmission of a single bit. \qed
\end{example}

The first perfect 1-error correcting codes were discovered in~1947 by
Hamming (but published in 1950)~\cite{Hamming}, and in~1949 by
Golay~\cite{Golay}, in two landmark papers in the field.
Since these works and another concurrent seminal work by
Shannon~\cite{Shannon}, a problem of considerable interest in
mathematics, information theory, and applications has been to try and
construct binary codes with large minimum distance. (For sphere packings,
in the sequel we will consider subsets/codes on the surface of a sphere,
termed \textit{spherical codes}.)

\subsection{Optimizing quadratic forms: from Lagrange to
Hermite, to lattice packings}\label{Shermite}

Before we return to Kepler's conjecture and sphere packings in other
dimensions, we explain how Lagrange's study of quadratic forms leads to
lattice packings, via work of Hermite (and via a constant that is named
after him). Some of this account is taken from~\cite[Section~1]{Nguyen}.

Given integers $a,b,c$, Lagrange was studying which integers can be
represented by quadratic forms $ax^2 + bxy + cy^2$ for $x,y \in
\mathbb{Z}$ -- inspired by previous work of Euler and of Fermat (who had
studied the ``sum of two squares'' case $a=c=1,b=0$). Thus, in his 1770s
work~\cite{Lagrange}, Lagrange generalized Euclid's algorithm to such
binary quadratic forms. This was further generalized in 1850 by
Hermite~\cite{Hermite} along a common theme to this survey and to
Theorem~\ref{Tbasics}: for $x = (x_1, \dots, x_n) \in \R^n$, define the
quadratic form $Q(x) := \sum_{i,j=1}^n q_{ij} x_i x_j$ for some real
symmetric \textit{positive definite} matrix/quadratic form $Q$. Hermite
showed:

\begin{theorem}[\cite{Hermite}]\label{Thermite}
Given a positive definite real matrix $Q_{n \times n}$, there exist
integers $x_1, \dots, x_n$ such that
\[
0 < Q((x_1, \dots, x_n)) \leq (4/3)^{(n-1)/2} \det(Q)^{1/n}.
\]
\end{theorem}

If we define $\| Q \| := \inf_{x \in \mathbb{Z}^n \setminus \{ 0 \}}
Q(x)$, it turns out that this infimum is positive and is attained. Thus
the ratio $\| Q \| / \det(Q)^{1/n}$ can be bounded above for all
positive definite $Q$. This yields the constant named after Hermite:

\begin{definition}
The \textit{$n$th Hermite constant} $\gamma_n := \displaystyle \sup_{Q =
Q^T \text{ positive definite}} \frac{\|Q\|}{\det(Q)^{1/n}}$. (This
supremum is in fact attained.)
\end{definition}

Moreover, the bivariate case of these facts had been shown by Lagrange:

\begin{theorem}[\cite{Lagrange}]\label{Tlagrange}
The supremum $\gamma_2$ exists/is attained, and equals $\sqrt{4/3}$.
\end{theorem}

It turns out that determining the Hermite constant is equivalent to a
restricted version of computing the packing density: doing so for
\textit{lattices}. Recall that a lattice in $\R^n$ is the
$\mathbb{Z}$-span of an $\R$-basis.

\begin{definition}
Given a lattice $L \subset \R^n$, its \textit{covolume} is the
$n$-dimensional volume of $\R^n / L$; and its \textit{least length}
$\lambda_1(L)$ is the length of any shortest nonzero element of $L$.
\end{definition}

The Hermite constant can be described in lattice-theoretic terms:

\begin{proposition}\label{Phermite}
$\gamma_n$ equals the supremum of $\lambda_1(L)^2$ over all unit covolume
lattices:
\begin{equation}
\gamma_n := \sup \{ \lambda_1(L)^2 \, : \, {\rm Vol}_{\R^n}( \R^n / L ) =
1 \}.
\end{equation}
\end{proposition}

\begin{proof}
At the outset, note that if $L$ is any lattice, then rescaling $L$ by
$c>0$ rescales the least length by $c$ and the covolume by $c^n$. Thus,
\begin{equation}\label{Elattice}
\frac{\lambda_1(cL)^2}{{\rm Vol}_{\R^n}(\R^n/cL)^{2/n}} = 
\frac{\lambda_1(L)^2}{{\rm Vol}_{\R^n}(\R^n/L)^{2/n}},
\end{equation}
and so we need to compare the supremum of this right-hand side to
$\gamma_n$.

This is done using a standard correspondence between quadratic forms $Q >
0$ and lattices $L$. Given $Q$, Theorem~\ref{Tbasics} yields an
invertible matrix $B_Q = [{\bf v}_1 | \cdots | {\bf v}_n]$ whose columns
have Gram matrix $Q$. This leads to the lattice $L({B_Q}) :=
\oplus_{i=1}^n \mathbb{Z} {\bf v}_i$.
Moreover, the choice of basis is unique up to an orthogonal change:
$B_Q^T B_Q = Q = C_Q^T C_Q$ if and only if $B_Q C_Q^{-1} \in O(n)$.

Now if $B_Q = UC_Q$ for orthogonal $U$, then we compute:
\begin{align}\label{Ehermite2}
\| Q \| = &\ \inf_{x \in \mathbb{Z}^n \setminus \{ 0 \}} x^T Q x = 
\inf_{x \in \mathbb{Z}^n \setminus \{ 0 \}} x^T B_Q^T B_Q x = \inf_{x \in
\mathbb{Z}^n \setminus \{ 0 \}} \left\| \sum_{i=1}^n x_i {\bf v}_i
\right\|^2 = \lambda_1(L)^2,\notag\\
\det(Q)^{1/n} = &\ \det (B_Q^T B_Q)^{1/n} = |\det(B_Q)|^{2/n} = {\rm
Vol}_{\R^n}(\R^n/L)^{2/n},
\end{align}
and both calculations are unchanged under $B_Q \to U B_Q = C_Q$.

Conversely, given a lattice $L$, choose an ordered basis that generates
it, say $({\bf v}_1, \dots, {\bf v}_n)$. Now define $B_L := [{\bf v}_1 |
\cdots | {\bf v}_n]$, and $Q_L := B_L^T B_L$. Also note that if $C_L$ is
any other generating basis then $C_L = P B_L$ for some matrix $P \in
GL_n(\mathbb{Z})$ (so $\det P = \pm 1$), and hence $Q \leadsto P^T Q P$.
Now $\| P^T Q P \| = \lambda_1(L)^2 = \| Q \|$ by~\eqref{Ehermite2},
since the lattice $L$ is unchanged; and $\det(P^T Q P) = \det(Q)$.
Moreover, one can reverse both calculations in~\eqref{Ehermite2}, with
$B_Q$ replaced by $B_L$.

It follows that the set of ratios in the definition of $\gamma_n$ (across
all $Q>0$) equals the set of ratios on the right side in~\eqref{Elattice}
(across all $L$). Taking suprema, the result follows.
\end{proof}

Having gone from quadratic forms to lattices, we now mention the
\textit{lattice packing problem}: given a dimension $n \geq 1$, compute
$\Delta^{(L)}_{\R^n}$. This turns out to be equivalent to the above:

\begin{proposition}
Given $n \geq 1$, computing the lattice packing density of $\R^n$ is
equivalent to determining the Hermite constant $\gamma_n$.
\end{proposition}

\begin{proof}
The result holds because of a simple equation governing the dependence
between three quantities:
(a)~the lattice packing density $\Delta^{(L)}_{\R^n}$,
(b)~the Hermite constant $\gamma_n$,
and (c)~the $n$-dimensional volume $\nu_n$ of the unit ball in $\R^n$ --
see \cite[Equation~(1)]{Bl29}:
\begin{equation}\label{Ehermite}
\gamma_n = 4 (\Delta^{(L)}_{\R^n} / \nu_n)^{2/n} \qquad
\Longleftrightarrow \qquad \Delta^{(L)}_{\R^n} = (\gamma_n/4)^{n/2} \nu_n
= (\gamma_n/4)^{n/2} \frac{\pi^{n/2}}{\Gamma(n/2 + 1)}.
\end{equation}

In fact this relationship holds for each lattice, and is not mysterious.
Define the \textit{packing radius} $r(L)$ to be the largest scalar such
that placing an $n$-sphere of this radius at each point of $L$ yields a
packing. Thus, $r(L) = \lambda_1(L)/2$. Now a ``fundamental domain'' for
$\R^n$ is the $n$-dimensional parallelopiped $\R^n/L$ (also called
``parallelotope''), and the space covered in it by the packing is
$2^n$-many spherical ``sectors/caps'' -- which make up exactly one sphere
of radius $\lambda_1(L)/2$, so of $n$-volume $\nu_n \cdot
(\lambda_1(L)/2)^n$. Thus, the density of the packing for $L$ is:
\[
\Delta_L := \frac{(\lambda_1(L)/2)^n}{{\rm Vol}_{\R^n} (\R^n/L)} \nu_n =
\frac{(\lambda_1(L)^2/4)^{n/2}}{{\rm Vol}_{\R^n} (\R^n/L)} \cdot
\frac{\pi^{n/2}}{\Gamma(n/2+1)}.
\]
Clearly, rescaling the lattice by $c$ changes both the numerator and
denominator by a factor of $c^n$, so we may normalize to assume that $L$
has covolume $1$. This is precisely~\eqref{Ehermite} before taking the
supremum, because of the proof of Proposition~\ref{Phermite} (and it
holds for every $L$ of covolume $1$), so now take the supremum.
\end{proof}

\subsection{Lattice packing densities in low dimensions; Hermite
constants}

We now return to the story of the Kepler conjecture, following Lagrange,
Gauss, Barlow, and Hilbert (see above). In 1953, L\'aszl\'o Fejes T\'oth
suggested a recipe for ascertaining Kepler's conjecture
(see~\cite{LFT53}), via checking a finite -- but very large -- number of
cases to solve a finite-variable optimization problem. This would require
advanced computing tools, which have since become available. In the 1990s
Hales, together with his student Ferguson, applied linear programming
techniques to try and minimize a function with 100+ variables, which they
had shown would suffice to compute the sphere packing density of $\R^3$.
Their research program took up 2+ years, 100,000 linear programming
problems, and 3 gigabytes of computer programs. The findings appeared in
the long articles~\cite{Hales,HF} in 2005--06. A decade later, Hales and
many coauthors completed and published a formal proof of the
computer-assisted calculations~\cite{Hales17}. (There was also a proof of
Kepler's conjecture by Wu-Yi Hsiang; but following criticism by G\'abor
Fejes T\'oth and others, the proof is currently regarded as incomplete.)

This ends the story of sphere packings in three dimensions; note that in
all cases discussed so far, the centers of the spheres in at least one
configuration in each dimension (which is the only one in $\R$ and
$\R^2$) lie on a lattice. These are precisely the lattices of Lie type
$A_n$ for $n=1,2,3$ -- i.e., the lattices generated in the hyperplane
$(1,\dots,1)^\perp \subset \R^{n+1}$ by the simple roots $\alpha_i = {\bf
e}_i - {\bf e}_{i+1}$ for $1 \leq i \leq n$.

From above, the search for the Hermite constant is the same as that for
the lattice packing density of Euclidean space. We have discussed the
$n=1,2,3$ cases above, and $\gamma_n, \Delta^{(L)}_{\R^n}$ are precisely
known for just six other values of $n$ to date. In Table~\ref{Tlpd} we
summarize these results -- via noting the values of $\gamma_n, \nu_n,
\Delta^{(L)}_{\R^n}$, the associated lattices (and all but one are of
simple Lie type), and their discoverers -- or the discoverers of
$\Delta^{(L)}_{\R^n}$, which is equivalent.

For completeness, we also mention the recent preprint \cite{Voronoi},
which computes the lattice packing density in nine dimensions -- or
equivalently, shows that the Hermite constant $\gamma_9 = 2$.

\begin{table}[ht]
\begin{tabular}{|c||c|c|c|c|c|c|c|c|c|}
\hline
$n$ & 1 & 2 & 3 & 4 & 5 & 6 & 7 & 8 & 24\\
\hline\hline
$\gamma_n^n$ & 1 & $\frac{4}{3}$ & 2 & 4 & 8 & $\frac{64}{3}$ & 64 &
$2^8$ & $4^{24}$\\
\hline
$\nu_n$ & 2 & $\pi$ & $\frac{4}{3} \pi$ & $\frac{1}{2} \pi^2$ &
$\frac{8}{15} \pi^2$ & $\frac{1}{6} \pi^3$ & $\frac{16}{105} \pi^3$ &
$\frac{1}{24} \pi^4$ & $\frac{1}{12!} \pi^{12}$\\
\hline
$\Delta^{(L)}_{\R^n}$ & 1 &
$\displaystyle \frac{\pi}{2\sqrt{3}}$ &
$\displaystyle \frac{\pi}{3\sqrt{2}}$ &
$\displaystyle \frac{\pi^2}{16}$ &
$\displaystyle \frac{\pi^2}{15\sqrt{2}}$ &
$\displaystyle \frac{\pi^3}{48\sqrt{3}}$ &
$\displaystyle \frac{\pi^3}{105}$ &
$\displaystyle \frac{\pi^4}{384}$ &
$\displaystyle \frac{\pi^{12}}{12!}$\\
\hline
Lattice & & & & & & & & & \\
(Lie) type & $A_1$ & $A_2$ & $A_3$ & $D_4$ & $D_5$ & $E_6$ & $E_7$ &
$E_8$ & Leech\\
\hline
& & & & \multicolumn{2}{c|}{Korkine,} & Hofreiter; &
\multicolumn{2}{c|}{} & Cohn,\\
By: & $-$ & Lagrange & Gauss & \multicolumn{2}{c|}{Zolotareff} &
Blichfeldt & \multicolumn{2}{c|}{Blichfeldt} & Kumar\\
\hline
Reference: & $-$ & \cite{Lagrange} & \cite{Gauss} & \cite{KZ72} &
\cite{KZ77} & \cite{Hofreiter}; \cite{Bl35} &
\multicolumn{2}{c|}{\cite{Bl35}} & \cite{CK09}\\
\hline
Year: & $-$ & 1773 & 1831 & 1872 & 1877 & 1933; 1935 &
\multicolumn{2}{c|}{1935} & 2009\\
\hline
\end{tabular}\vspace*{2mm}
\caption{The known Hermite constants and lattice packing
densities}\label{Tlpd}
\end{table}

\subsection{Kissing numbers: exact answers and bounds}

Having discussed exact answers in a few dimensions for the lattice
packing density, before moving to all packings we take a detour into a
related problem: determining the \textit{kissing number} in $\R^n$. This
is the largest number of non-overlapping unit spheres that a unit sphere
can simultaneously touch, or ``kiss'' tangentially, and we will denote it
by $k(n)$. It is not hard to show that $k(1) =2$ and $k(2) =6$. The $n=3$
case was the subject of a famous 1694 debate between Newton (who thought
it was $12$) and Gregory (who thought it was $13$). Thus, this question
is also called the \textit{thirteen spheres problem}, and the integer of
interest is also called the \textit{Newton number} or \textit{contact
number}. It was solved more than 250 years later -- in 1952 -- by
Sch\"utte and van der Waerden~\cite{SvdW}, who showed that $k(3)$ is
indeed $12$.

In four dimensions, the kissing number was computed by Musin in 2008 to
be~$24$~\cite{Musin}. The only other kissing numbers that are known (see
Table~\ref{Tkissing}) are in dimensions $n=8, 24$ -- and they were both
computed independently in~1979 by Levenshtein~\cite{Levenshtein} and by
Odlyzko--Sloane~\cite{OS}: $k(8) = 240$ and $k(24) = 196560$. Remarkably,
these two numbers also arise from spheres with centers in the same two
lattices as for sphere packings! We will sketch their proofs below.

\begin{table}[ht]
\begin{tabular}{|c||c|c|c|c|c|c|c|c|c|}
\hline
$n$ & 1 & 2 & 3 & 4 & 8 & 24\\
\hline\hline
$k(n) = A(n,\pi/3)$ & 2 & 6 & 12 & 24 & 240 & $196560$\\
\hline
Lattice & $A_1$ & $A_2$ & $H_3$ & $D_4$ & $E_8$ & Leech\\
(Coxeter) type & & & & & & \\
\hline
By: & $-$ & $-$ & Sch\"utte, & Musin & \multicolumn{2}{c|}{Levenshtein;}\\
& & & van der Waerden & & \multicolumn{2}{c|}{Odlyzko--Sloane}\\
\hline
Reference: & $-$ & $-$ & \cite{SvdW} & \cite{Musin} &
\cite{Levenshtein}; \cite{OS} & \cite{Levenshtein}; \cite{OS}\\
\hline
Year: & $-$ & $-$ & 1952 & 2003 & 1979 & 1979\\
\hline
\end{tabular}\vspace*{2mm}
\caption{The known kissing numbers}\label{Tkissing}
\end{table}

In addition to these exact results, and decades before Musin's 2008
article, upper and lower bounds were sought -- for general $n$, not
specific values. (This will also be the theme when we consider the
(lattice) packing density of $\R^n$ for general $n$.) E.g.,
Coxeter~\cite{Coxeter} proposed some upper bounds in 1963. Wyner provided
in~1965 an asymptotic lower bound of $2^{0.2075n (1+o(1))}$~\cite{Wyner}.
But the object of our focus here and below is a $\sim$50-year old upper
bound due to Kabatiansky and Levenshtein~\cite{KL}.

\begin{definition}
Given an angle $\psi \in [0,\pi]$ and a dimension $n \geq 2$, a finite
subset $X \subset S^{n-1}$ is a \textit{spherical $\psi$-code} if the
spherical distance between any two vectors in $X$ is at least $\psi$:
\[
\sphericalangle x,y \ \geq \psi \qquad \Longleftrightarrow \qquad
\tangle{x,y} \leq \cos \psi.
\]
Let $A(n,\psi)$ denote the size of any maximum-cardinality spherical
$\psi$-code in $S^{n-1}$.
\end{definition}

\begin{remark}
The name ``code'' is akin to that in ``error-correcting code'' -- see
Example~\ref{ExHamming} -- in that it too stands for a set of points in a
metric space, with all nonzero distances uniformly bounded below.

Additionally, spherical codes have their origins in a well-known question
in mathematical biology: the \textit{Tammes problem}, formulated in~1930
by Tammes~\cite{Tammes}. The problem asks: given integers $n,N \geq 2$,
pack $N$ points on $S^{n-1}$ such that the minimum distance between
distinct points gets maximized.
\end{remark}

As far as sphere packings and kissing numbers go, note that if two
non-overlapping unit spheres kiss a common unit sphere, the closest that
their centers can get is precisely when the three centers form an
equilateral triangle. Thus, every packing or kissing problem involves
$\psi \geq \pi/3$, i.e., $\cos \psi \leq 1/2$. In particular, the kissing
number is
\begin{equation}
k(n) = A(n,\pi/3), \qquad \forall n \geq 2.
\end{equation}

Now the celebrated 1978 upper bound of Kabatiansky--Levenshtein is:

\begin{theorem}[{\cite[Corollary~1]{KL}}]\label{TKL}
For all $\psi \in (0,63^\circ)$, we have
\[
n^{-1} \log_2 A(n,\psi) \leq \frac{-1}{2} \log_2(1-\cos \psi) - 0.099 +
o(1).
\]
In particular, $k(n) = A(n,\pi/3) \leq 2^{n(0.401 + o(1))}$.
\end{theorem}

\subsection{The $E_8$ and Leech lattices}\label{Slattices}

Notice in the aforementioned results -- on both the lattice packing and
kissing number problems -- that beyond the first three or four
dimensions, two dimensions stood out: $n=8, 24$. This is because in both
of these dimensions, both of the above problems get solved when the
centers lie in the same, remarkable lattice of rank $8$ or $24$. Thus,
for completeness we write down characterizations for both lattices.

The $E_8$ lattice is widely studied in the context of Lie theory,
mathematical and particle physics, and Hamming codes (among other areas).
It is a rank 8 lattice with any of the following properties:
\begin{enumerate}
\item It is at once integral, even, and unimodular. Namely, $x \cdot y$
is an integer for all $x,y \in E_8$; $x \cdot x$ is moreover even for all
$x$; and $E_8$ has a covolume $1$.

\item An alternate description is the set of points $x = (x_1, \dots,
x_8) \in \mathbb{R}^8$ such that $\sum_i x_i$ is an even integer, and the
$x_i$ are either all integers or all half-integers.
\end{enumerate}

$E_8$ is also the root lattice for the largest exceptional complex simple
Lie algebra; in Figure~\ref{FigE8} we provide its Dynkin diagram.

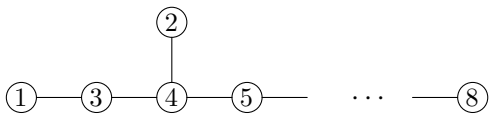
\begin{figure}[ht]
\begin{tikzpicture}[line cap=round,line join=round,>=triangle 45,x=1.0cm,y=1.0cm]
\draw(1,12) circle (0.2cm);
\draw(2,12) circle (0.2cm);
\draw(3,12) circle (0.2cm);
\draw(4,12) circle (0.2cm);
\draw(7,12) circle (0.2cm);
\draw(3,13) circle (0.2cm);
\draw (0.77,12.25) node[anchor=north west] {{\small 1}};
\draw (1.77,12.25) node[anchor=north west] {{\small 3}};
\draw (2.77,12.25) node[anchor=north west] {{\small 4}};
\draw (3.77,12.25) node[anchor=north west] {{\small 5}};
\draw (5.25,12.2) node[anchor=north west] {$\cdots$};
\draw (6.77,12.25) node[anchor=north west] {\small{8}};
\draw (2.77,13.25) node[anchor=north west] {{\small 2}};
\draw [-] (1.2,12) -- (1.8,12);
\draw [-] (2.2,12) -- (2.8,12);
\draw [-] (3.2,12) -- (3.8,12);
\draw [-] (4.2,12) -- (4.8,12);
\draw [-] (6.2,12) -- (6.8,12);
\draw [-] (3,12.2) -- (3,12.8);
\end{tikzpicture}
\caption{The nodes (or simple roots) of the $E_8$ Dynkin
diagram}\label{FigE8}
\end{figure}

It is known that the norm-square $Q : E_8 \to \R$, $Q(x) := \| x \|^2$ is
a positive definite quadratic form. (Thus, up to isometry, $E_8$ is also
the unique even, unimodular, positive definite lattice of rank $8$.) This
connects $E_8$ to both themes in this article: positivity and lattice
packings. Indeed, the existence of such a form was proved in the very
first article cited in this article, in~1868 by Smith~\cite{Smith}. It
was then explicitly constructed first by Korkine--Zolotareff~\cite{KZ73}
in~1873, in the same series of works in which they computed the Hermite
constant in dimensions~4 and~5 (see above). In~1938,
Mordell~\cite{Mordell1} showed the uniqueness of a lattice with the above
properties (he wrote a related work on lattice packings and $\gamma_n$,
mentioned below).

\begin{remark}\label{RE8kissing}
Moreover, one checks that the shortest norm-square of a nonzero vector in
$E_8$ is $\lambda_1(E_8)^2 = 2$, and there are precisely $240$ such
vectors ``nearest'' to the origin. Thus, if one places spheres of packing
radius $r(E_8) = \lambda_1(E_8)/2$ at these lattice points, they kiss the
congruent sphere centered at the origin. Thus, $k(8) \geq 240$.
\end{remark}

We next mention the \textit{Leech lattice} $\Lambda_{24}$, which also
features outside sphere packings in geometry (including
higher-dimensional versions of the Tammes problem in mathematical
biology), modular forms, group theory (via Conway groups), coding theory,
and moonshine theory and vertex operator algebras. This lattice was
introduced in~1967 by Leech \cite{Leech}, following lifting from
$(\mathbb{Z}/2\mathbb{Z})^{24}$ to $\mathbb{Z}^{24}$ the extended Golay
code \cite{Golay} -- this has minimal Hamming distance 8 (see
Example~\ref{ExHamming}). The Leech lattice is the unique (up to
isometry) unimodular even lattice in $\R^{24}$, and it has least length
$2$. Once again, one checks that there are 196560 points in
$\Lambda_{24}$ of least length (i.e., closest to the origin). Thus,
$k(24) \geq 196560$.

\subsection{Kissing numbers and spherical codes, via Delsarte -- and
Schoenberg}\label{Sdelsarte}

Here we elaborate on the \textit{linear programming method} that led to
the Kabatiansky--Levenshtein asymptotic upper bound for the kissing
number $k(n)$. This method was pioneered by Delsarte in the 1970s --
originally in \cite{Delsarte1} for cardinalities of binary codes; then
in~\cite{Delsarte2} for more general ``association schemes''; and finally
in 1977 with Goethals and Seidel~\cite{DGS}, where Gegenbauer polynomials
enter into the picture. This is the promised connection to Schoenberg's
work (see Theorem~\ref{Tschoenberg-pd}) and to the above survey on
positivity preservers.

\subsubsection{Refresher on Gegenbauer polynomials}

We recall here some basics on Gegenbauer polynomials $G^{(n)}_k(t)$,
which formed the basis (literally!) of Schoenberg's classification of
positive definite functions over spheres $S^{n-1}$. These can be defined
in multiple ways: for instance, via their generating function
\[
(1 - 2rt + r^2)^{(2-n)/2} = \sum_{k=0}^\infty r^k C_k^{(n)}(t); \qquad
G_k^{(n)}(t) := \frac{C_k^{(n)}(t)}{C_k^{(n)}(1)}
\]
if $n \geq 3$ -- and for $n=2$, we have:
\[
\frac{1-rt}{1-2rt+r^2} = \sum_{k=0}^\infty r^k G^{(2)}_k(t).
\]

A second recipe to define these polynomials is via a three term
recurrence, for any $n \geq 2$:
\[
G_0^{(n)}(t) = 1, \quad G_1^{(n)}(t) = t, \qquad G_k^{(n)}(t) =
\frac{(2k+n-4)t G_{k-1}^{(n)}(t) - (k-1) G_{k-2}^{(n)}(t)}{k+n-3}\
\forall k \geq 2.
\]

A third description is that the $G_k^{(n)}(t)$ are degree $k$ polynomials
that are normalized -- $G_k^{(n)}(1) = 1$ -- and form an orthogonal
system with respect to integrating on $[-1,1]$ against the measure
$(1-t^2)^{(n-3)/2}\, dt$. (This is the projection to $[-1,1]$ of the
surface measure $d\omega_{n-1}$ of the sphere.)

These orthogonal polynomials subsume various special cases. Setting
$n=2,3,4$, we recover, respectively: Chebyshev polynomials of the first
kind, Legendre polynomials, and Chebyshev polynomials of the second kind.
Moreover, while the above are perfectly adequate, self-contained
definitions/characterizations of Gegenbauer polynomials, we now provide a
fourth, beautiful description, wherein they arise naturally through
spherical harmonics.

\subsubsection{Spherical harmonics and the Addition Theorem}

The following account is taken from the survey~\cite{PZ} and the early
part of M\"uller's notes~\cite{Muller-book}.

We work over the sphere $S^{n-1} \subset \R^n$. Thus, for $0 \neq x \in
\R^n$, write $x = \|x\| \xi$, with $\xi \in S^{n-1}$. The area element on
the sphere $S^{n-1}$ will be denoted by $d\omega_{n-1} = d
\omega_{n-1}(\xi)$, and the $(n-1)$-dimensional ``surface area'' of the
sphere is
\[
\omega_{n-1} := \int_{S^{n-1}} d \omega_{n-1} = \frac{2
\pi^{n/2}}{\Gamma(n/2)}.
\]
Thus $\omega_0 = |S^0| = 2$, $\omega_1 = 2 \pi$, $\omega_2 = 4 \pi$, and
so on.

Next, we have the \textit{Laplace operator}
$\displaystyle \Delta_n := \sum_{j=1}^n \frac{\partial^2}{\partial
x_j^2}$, and its null vectors among polynomials:

\begin{definition}
Given integers $n \geq 2$ and $k \geq 0$, a \textit{spherical harmonic}
in $n$ dimensions of degree/order $k$ is a harmonic (i.e.\ $\Delta_n f
\equiv 0$) polynomial $H_k(x) = H_k((x_1, \dots, x_n))$ that is
homogeneous of degree $k$, and is now restricted to $S^{n-1}$. We will
denote these by $H_k(\xi)$. Let $\mathcal{SH}_{n,k}$ denote the space of
such polynomials.
\end{definition}

The above integral defines an inner product on the span of all spherical
harmonics (of all degrees) $\bigoplus_{k \geq 0} \mathcal{SH}_{n,k}$:
\begin{equation}
\tangle{f,g} := \int_{S^{n-1}} f(\xi) g(\xi) \, d \omega_{n-1}.
\end{equation}

By Green's theorem, one checks that for nonnegative integers $k \neq k'$
and corresponding degree spherical harmonics $H_k, H_{k'}$,
\begin{align*}
0 = \int_{\|x\| \leq 1} (H_k \Delta_n H_{k'} - H_{k'} \Delta_n H_k) \, dx
= &\ \int_{S^{n-1}} \left( H_k \frac{\partial H_{k'}}{\partial r} -
H_{k'} \frac{\partial H_k}{\partial r} \right) d \omega_{n-1}\\
= &\ \int_{S^{n-1}} (k'-k) H_k(\xi) H_l(\xi)\, d \omega_{n-1},
\end{align*}
since $H_k$ has normal derivative on $S^{n-1}$ (in the $r = \|x\|$
direction) equal to
\[
\left. \frac{\partial}{\partial r} H_k(r\xi) \right|_{r=1} = \left. k
r^{k-1} H_k(\xi) \right|_{r=1} = k \, H_k(\xi), \quad \forall k \geq 0.
\]

Thus, spherical harmonics of differing degrees are orthogonal -- and for
a given degree $k \geq 0$, one uses e.g.\ Gram--Schmidt to obtain an
orthonormal basis of $\mathcal{SH}_{n,k}$. One also has:
\begin{equation}\label{Edim}
N(n,k) := \dim \mathcal{SH}_{n,k} = \binom{k+n-2}{k} +
\binom{k+n-3}{k-1}, \qquad \forall n \geq 2,\, k \geq 0.
\end{equation}

\begin{example}[$n=2$]
We consider a well known special case, where $n=2$ and the sphere is the
unit circle. In this case $\dim \mathcal{SH}_{2,k} = 2$ for $k>0$, and
$\mathcal{SH}_{2,0} = \R \cdot 1$. For $k>0$, two linearly independent
degree $k$ spherical harmonics are $\Re \, (x_2 + i x_1)^k$ and $\Im \,
(x_2+ix_1)^k$. Now introduce polar coordinates: $x_1 = r \cos \theta, x_2
= r \sin \theta$. Using this, we obtain an orthonormal set of spherical
harmonics for each $k>0$:
\begin{align*}
S_{k,1} := &\ \frac{1}{\sqrt{\pi} \, r^k} \Re(x_2 + i x_1)^k =
\frac{1}{\sqrt{\pi}} \cos k(\textstyle \frac{\pi}{2} - \theta), \\
S_{k,2} := &\ \frac{1}{\sqrt{\pi}\, r^k} \Im(x_2 + i x_1)^k =
\frac{1}{\sqrt{\pi}} \sin k(\textstyle \frac{\pi}{2} - \theta).
\end{align*}
\end{example}

This example shows that when one thinks of $S^1$ not as a torus but as a
one-dimensional sphere, the way to generalize Fourier series to higher
dimensions is via spherical harmonics.

Now we bring in the orthogonal group.

\begin{lemma}\label{LSH}
If $u : \R^n \to \R$ is smooth and $A \in O(n)$, then $\Delta_n (u \circ
A) \equiv (\Delta_n u) \circ A$ on $\R^n$. Thus, if $H_k(\xi)$ is a
spherical harmonic then so is $H_k(A \xi)$.
\end{lemma}

\begin{proof}
The first assertion is an explicit computation, and immediately yields
the second.
\end{proof}

\begin{proposition}\label{Pspherical}
Fix integers $n \geq 2$ and $k \geq 0$ and an orthogonal matrix $A \in
O(n)$. Suppose $\{ S_{n,l} : 1 \leq l \leq N(n,k) \}$ is an orthonormal
basis of $(\mathcal{SH}_{n,k}, \tangle{\cdot,\cdot})$.
\begin{enumerate}
\item Then so is $\{ \xi \mapsto S_{n,l}(A \xi) : 1 \leq l \leq N(n,k)
\}$.

\item Define the matrix $C^{(A)}$ via expanding $S_{n,l}(A \cdot -)$ in
the orthonormal basis above:
\[
S_{n,l}(A \xi) = \sum_{r=1}^{N(n,k)} c_{lr}^{(A)} S_{n,r}(\xi), \qquad
C^{(A)} := (c_{lr}^{(A)})_{l,r=1}^{N(n,k)}.
\]
Then $C^{(A)}$ is also orthogonal, and $A \mapsto C^{(A)}$ is a group
homomorphism $: O(n) \to O(N(n,k))$.
In other words, $\mathcal{SH}_{n,k}$ is a finite-dimensional unitary
representation of $O(n)$, under $A \cdot S_{n,l}(\xi) := S_{n,l}(A^T
\xi)$.

\item The kernel
\begin{equation}\label{Egeg}
F : S^{n-1} \times S^{n-1} \to \R; \qquad F(\xi, \eta) :=
\sum_{l=1}^{N(n,k)} S_{n,l}(\xi) S_{n,l}(\eta)
\end{equation}
is invariant under the diagonal action of $O(n)$ on its arguments, and
hence depends only on their cosine $t = \tangle{\xi,\eta}$.
\end{enumerate}
\end{proposition}

Note that the group map $A \mapsto C^{(A)}$ need not be injective. For
instance, let $n$ be even and $A = -{\rm Id}_n$. Then $S_{n,l}(A \xi) =
S_{n,l}(\xi)$ for all $k \geq 0$, $1 \leq l \leq N(n,k)$, and $\xi \in
S^{n-1}$.

\begin{proof}\hfill
\begin{enumerate}
\item Fix $A \in O(n)$ and define $T_{n,l}(\xi) := S_{n,l}(A \xi)$. This
is also in $\mathcal{SH}_{n,k}$ by Lemma~\ref{LSH}, so we have structure
constants:
\[
S_{n,l}(A \xi) = T_{n,l}(\xi) = \sum_{r=1}^{N(n,k)} c^{(A)}_{lr}
S_{n,r}(\xi).
\]

Now compute $\tangle{T_{n,l}, T_{n,m}}$ in two ways: as an inner product
and as an integral. First, since the $S_{n,l}$ are orthonormal,
$\displaystyle \tangle{T_{n,l}, T_{n,m}} = \sum_{r=1}^{N(n,k)}
c_{lr}^{(A)} c_{mr}^{(A)}$.
Second, since $S^{n-1}$ and its surface measure $d \omega_{n-1}$ are
invariant under the orthogonal group,
\[
\tangle{T_{n,l}, T_{n,m}} = \int_{S^{n-1}} S_{n,l}(A \xi) S_{n,m}(A \xi)
\, d \omega_{n-1} = \int_{S^{n-1}} S_{n,l}(\xi) S_{n,m}(\xi) \, d
\omega_{n-1} = \delta_{l,m}.
\]
Equating the two expressions proves the assertion -- and also shows that
the matrix $C^{(A)} = (c_{lr}^{(A)})$ is orthogonal: $C^{(A)} (C^{(A)})^T
= {\rm Id}_{N(n,k)}$.

\item It suffices to show $A \mapsto C^{(A)}$ is multiplicative. This is
just a formal exercise -- given $A,B \in O(n)$, we have:
\begin{align*}
S_{n,l}(AB\xi) = &\ \sum_{r=1}^{N(n,k)} c^{(A)}_{lr} S_{n,r}(B \xi) =
\sum_{r=1}^{N(n,k)} c^{(A)}_{lr} \sum_{q=1}^{N(n,k)} c^{(B)}_{rq}
S_{n,q}(\xi)\\
= &\ \sum_{q=1}^{N(n,k)} \left( \sum_{r=1}^{N(n,k)} c_{lr}^{(A)}
c_{rq}^{(B)} \right) S_{n,q}(\xi). 
\end{align*}
On the other hand,
$\displaystyle S_{n,l}(AB\xi) = \sum_{q=1}^{N(n,k)} c^{(AB)}_{lq}
S_{n,q}(\xi)$. Hence $C^{(AB)} = C^{(A)} C^{(B)}$. The translation into
$\mathcal{SH}_{n,k}$ being a unitary representation is now another formal
exercise.

\item To show~\eqref{Egeg}, first note that $F$ is invariant under the
diagonal action of $O(n)$:
\[
F(A \xi, A \eta) = \sum_{l=1}^{N(n,k)} \sum_{m,q=1}^{N(n,k)}
c^{(A)}_{lm} c^{(A)}_{lq} S_{n,m}(\xi) S_{n,q}(\eta)
= \sum_{m,q=1}^{N(n,k)} S_{n,m}(\xi) S_{n,q}(\eta) \sum_{l=1}^{N(n,k)}
c^{(A)}_{lm} c^{(A)}_{lq}.
\]
But the inner sum is the $(m,q)$th entry of $(C^{(A)})^T C^{(A)} = {\rm
Id}_{N(n,k)}$, so we get
\[
F(A \xi, A \eta) = F(\xi, \eta) \qquad \forall \xi, \eta \in S^{n-1}, \,
A \in O(n).
\]

In particular, this invariance holds for $A \in SO(n)$. We are now done
by Lemma~\ref{Lsame}, since $K = F$ is clearly continuous.\qedhere
\end{enumerate}
\end{proof}

With these preliminaries on spherical harmonics and Gegenbauer
polynomials, one can identify exactly what is the function of the cosine
in~\eqref{Egeg}:

\begin{theorem}[{Addition Theorem,
\cite[Equation~(3.18)]{Muller}}]\label{Tadd}
Let $n,k$, and $\{ S_{n,l} : 1 \leq l \leq N(n,k) \}$ be as in
Proposition~\ref{Pspherical}.
Then the function in~\eqref{Egeg} is the rescaled $k$th Gegenbauer
polynomial:
\begin{equation}
\sum_{l=1}^{N(n,k)} S_{n,l}(\xi) S_{n,l}(\eta) =
\frac{N(n,k)}{\omega_{n-1}} G_k^{(n)}(\tangle{\xi,\eta}),
\end{equation}
where $\omega_{n-1} = \displaystyle \frac{2 \pi^{n/2}}{\Gamma(n/2)}$ is
the surface area of $S^{n-1}$.
\end{theorem}

\begin{remark}
This result likely first appeared in work of M\"uller~\cite{Muller}; but
at the start of this paper he attributes the ideas in the entire paper to
a lecture given by Herglotz to the G\"ottingen Mathematical Society on
November 1, 1945. M\"uller again credits Herglotz for the Addition
Theorem in his book~\cite{Muller-book}.
\end{remark}

\begin{example}
Before proceeding further, we write down the Addition Theorem in the
special case discussed above: $n=2$. Write
$\xi = e^{i \theta}, \ \eta = e^{i \phi}$,
so that
\[
\tangle{\xi , \eta} = \cos \theta \cos \phi + \sin \theta \sin \phi =
\cos(\theta - \phi),
\]
which is the addition formula for $\cos(\cdot)$. Now 
\[
S_{2,1}(\xi) S_{2,1}(\eta) + S_{2,2}(\xi) S_{2,2}(\eta) = \frac{1}{\pi}
\cos(n (\theta - \phi)).
\]
Thus, as a function of $t = \tangle{\xi , \eta}$, the Addition Theorem
specializes to yield:
\[
t = \cos(\theta - \phi) \mapsto \cos(n(\theta - \phi)),
\]
which is precisely the Chebyshev polynomial (in $t$) of the first kind.
\qed
\end{example}

\subsubsection{The easier half of Schoenberg's theorem on Gegenbauer
polynomials}

We now come to Schoenberg's theorem~\ref{Tschoenberg-pd}. For spherical
code bounds like the one by Kabatiansky--Levenshtein, and to compute the
kissing numbers in dimensions $8$ and $24$, we only need the ``easier''
implication, and this follows quickly from the Addition Theorem:

\begin{proof}[Proof of Theorem~\ref{Tschoenberg-pd}, easier half]
The claim to be proved is:
\textit{any $\R_+$-linear combination of Gegenbauer polynomials
$G_k^{(n)} \circ \cos$ is positive definite on distance matrices of
$S^{n-1}$.}

To show this, we can remove the $\cos(\cdot)$ and replace distance
matrices by their entrywise cosines, aka Gram matrices drawn from
$S^{n-1}$. Now it suffices to work with a single Gegenbauer polynomial
$G_k^{(n)}$. But by the Addition Theorem~\ref{Tadd}, given any vectors
$\xi_1, \dots, \xi_N \in S^{n-1}$ and scalars $x_1, \dots, x_N \in \R$,
we have:
\begin{align}\label{Eharmonics}
x^T \cdot ( G_k^{(n)}(\tangle{\xi_i, \xi_j}) )_{i,j=1}^N \cdot x = &\
\sum_{i,j=1}^N x_i x_j G_k^{(n)}(\tangle{\xi_i, \xi_j})\\
= &\
\sum_{i,j=1}^N \sum_{l=1}^{N(n,k)} \frac{\omega_{n-1}}{N(n,k)} x_i x_j
S_{n,l}(\xi_i) S_{n,l}(\xi_j)\notag\\
= &\ \frac{\omega_{n-1}}{N(n,k)} \sum_{l=1}^{N(n,k)} \sum_{i,j=1}^N x_i
x_j S_{n,l}(\xi_i) S_{n,l}(\xi_j)\notag\\
= &\ \frac{\omega_{n-1}}{N(n,k)} \sum_{l=1}^{N(n,k)} \left( \sum_{i=1}^N
x_i S_{n,l}(\xi_i) \right)^2 \geq 0,\notag
\end{align}
and so $G_k^{(n)}[-]$ indeed sends Gram matrices from $S^{n-1}$ to
positive semidefinite matrices.
\end{proof}

\begin{remark}\label{Rgegenbauer}
In fact, we do not even need the full power of the ``easier half'' of
Schoenberg's theorem. Below, we will only need that the sum of the matrix
entries $\sum_{i,j=1}^N G_k^{(n)}( \tangle{\xi_i, \xi_j} )$ is
nonnegative. But this follows by setting all $x_i = 1$
in~\eqref{Eharmonics}.
\end{remark}

We conclude with some remarks on Schoenberg's 1942
paper~\cite{Schoenberg42} and a footnote in it. First, to show the above
easier half of Theorem~\ref{Tschoenberg-pd}, Schoenberg did not use the
Addition Theorem as it was unavailable at the time; instead, he used an
inductive and intrinsic Addition Formula for Gegenbauer polynomials $\{
G_k^{(n)} : k \geq 0 \}$ in terms of $\{ G_k^{(n-1)} : k \geq 0 \}$.

Second, there is a path from Schoenberg's results in~\cite{Schoenberg42}
to not just regularization of covariance matrices, but also to a
celebrated conjecture in complex analysis. Note that $G_k^{(n)} \circ
\cos$ is positive definite on $S^{n-1}$ for all $k$, hence on $S^{n-2}$.
Hence by the ``harder half'' of Schoenberg's
theorem~\ref{Tschoenberg-pd}, $G_k^{(n)}$ must be a nonnegative real
combination of lower order Gegenbauer polynomials $\{ G_k^{(n-1)} : k
\geq 0 \}$. Schoenberg noted this in \cite[Footnote~2]{Schoenberg42}, and
then remarked that this should hold when one replaces the integer
parameter $n \geq 2$ by $\rho+2$, for a nonnegative real parameter
$\rho$. It later emerged that this had already been worked out in~1884 by
Gegenbauer himself~\cite{Gegenbauer}. Almost a century later, in~1976
Askey and Gasper used these results to show in~\cite{AG} the
nonnegativity of a sequence of generalized hypergeometric functions
${}_3F_2$. These ideas were subsequently used in~1985 by de Branges, in
his famous resolution of the Bieberbach conjecture~\cite{deBranges}.

\subsubsection{Application: kissing numbers are attained on $E_8$ and
$\Lambda_{24}$, via the Delsarte--Goethals--Seidel bound}

To take stock of the last few pages: we have come from spherical
harmonics, to the Addition Theorem, to the easier half of Schoenberg's
theorem~\ref{Tschoenberg-pd} on positive definite functions on spheres,
aka positivity preservers on correlation matrices with rank bounded
above. We now use this result to compute the kissing numbers $k(8)$ and
$k(24)$, via Delsarte's method. In fact we show a stronger statement,
which involves one more notion.

\begin{definition}
Given an integer $n \geq 1$, let $k^{(L)}(n)$ denote the {\em lattice
kissing number of $\R^n$}, i.e., the largest number of unit spheres that
touch $S^{n-1}$, and whose centers lie on a lattice $L$ with
$\lambda_1(L) = 2$.
\end{definition}

The following observations are clear. First, we have $2n \leq k^{(L)}(n)
\leq k(n)$ for all $n$, since one can place sphere-centers at $\{ \pm 2
{\bf e}_j : 1 \leq j \leq n \}$ to kiss the unit sphere around the
origin. (Thus, $L = \oplus_{j=1}^n \mathbb{Z} (2 {\bf e}_j)$.)
Second, if $k(n)$ is attained at a lattice, then $k(n) = k^{(L)}(n)$.
This is indeed the case for $n = 1,2,3,4$.
Now one can show:

\begin{theorem}[\cite{Levenshtein,OS}]\label{TLP}
We have $k(8) = k^{(L)}(8) = 240$ (attained on $E_8$)
and $k(24) = k^{(L)}(24) = 196560$ (attained on $\Lambda_{24}$).
\end{theorem}

The proof uses a famous upper bound on spherical codes, by
Delsarte--Goethals--Seidel:

\begin{theorem}[\cite{DGS}]\label{TDGS}
Fix an integer $n \geq 2$ and an angle $\psi \in (0,\pi]$. Let $f(t) =
\displaystyle \sum_{k=0}^\infty c_k G_k^{(n)}(t)$, where $n \geq 2$ and
all $c_k \in [0,\infty)$. Further assume that $c_0 > 0$ and $f(t) \leq 0$
for all $t \in [-1,\cos \psi]$. Then we have the upper bound
\begin{equation}
A(n,\psi) \leq \frac{f(1)}{c_0}.
\end{equation}
\end{theorem}

As promised, this will use the entrywise calculus and Schoenberg's
positivity preserver $f$!

\begin{proof}
Choose $N := A(n,\psi)$-many points $\xi_1, \dots, \xi_N$ on the sphere
with pairwise angles at least $\psi$. Apply the entrywise map $f[-]$ to
their Gram matrix. Then,
\[
\sum_{i,j=1}^N f(\tangle{\xi_i, \xi_j}) =
\sum_{k=0}^\infty c_k \sum_{i,j=1}^N G_k^{(n)}(\tangle{\xi_i, \xi_j})
\geq c_0 \sum_{i,j=1}^N G_0^{(n)}(\tangle{\xi_i, \xi_j}) = c_0 N^2,
\]
using Remark~\ref{Rgegenbauer} and that $G_0 \equiv 1$.

Also note that by hypothesis, $f$ applied to any off-diagonal entry is
non-positive, since distinct points have inner product at most
$\cos(\psi)$. Therefore,
\[
\sum_{i,j=1}^N f(\tangle{\xi_i, \xi_j}) = N f(1) + \sum_{i \neq j}
f(\tangle{\xi_i,\xi_j}) \leq N f(1).
\]
Combining the two bounds gives $A(n,\psi) = N \leq f(1)/c_0$.
\end{proof}

Finally, we employ Theorem~\ref{TDGS} to show:

\begin{proof}[Proof of Theorem~\ref{TLP}, taken from~\cite{OS}]
We first prove the $n=8$ case. We in fact showed (but did not state) in
Section~\ref{Slattices} that
$k(8) \geq k^{(L)}(8) \geq 240$.
Now consider the carefully chosen degree 6 polynomial
\begin{align}
f(t) = &\ \frac{320}{3} (t+1) (t + \textstyle \frac{1}{2})^2 t^2 (t -
\textstyle \frac{1}{2})\\
= &\ G_0^{(8)} +
\frac{16}{7} G_1^{(8)} +
\frac{200}{63} G_2^{(8)} +
\frac{832}{231} G_3^{(8)} +
\frac{1216}{429} G_4^{(8)} +
\frac{5120}{3003} G_5^{(8)} +
\frac{2560}{4641} G_6^{(8)},\notag
\end{align}
where we suppress the ``$(t)$'' from the Gegenbauer polynomials,
and one solves for the coefficients of the polynomials $G_k^{(8)}$ by a
triangular change of basis from the monomial basis.

Then $f \leq 0$ on $[-1,1/2]$, so one can apply Theorem~\ref{TDGS} with
$\psi = \pi/3$ to upper bound the kissing number:
\[
k(8) = A(8,\pi/3) \leq f(1) / c_0 = 240.
\]
Combined with the above lower bounds, we are done by sandwiching.

The proof is similar for $n=24$: we had seen that $k(24) \geq k^{(L)}(24)
\geq 196560$. Now consider
\begin{align}
f(t) = &\ \frac{1490944}{15} (t+1) (t + \textstyle \frac{1}{2})^2 (t +
\textstyle \frac{1}{4})^2 t^2 (t - \textstyle \frac{1}{4})^2 (t -
\textstyle \frac{1}{2})\\
= &\ G_0^{(24)} +
\frac{48}{23} G_1^{(24)} +
\frac{1144}{425} G_2^{(24)} +
\frac{12992}{3825} G_3^{(24)} +
\frac{73888}{22185} G_4^{(24)} \notag\\
&\ + \frac{2169856}{687735} G_5^{(24)}
+ \frac{59062016}{25365285} G_6^{(24)}
+ \frac{4472832}{2753575} G_7^{(24)} \notag\\
&\ + \frac{23855104}{28956015} G_8^{(24)}
+ \frac{7340032}{20376455} G_9^{(24)}
+ \frac{7340032}{80848515} G_{10}^{(24)}. \notag
\end{align}
Once again $f \leq 0$ on $[-1,1/2]$, so we apply Theorem~\ref{TDGS} with
$\psi = \pi/3$ to obtain $k(24) \leq 196560$. Combined with the above
lower bounds, we are again done.
\end{proof}

\subsection{From spherical codes to sphere packing upper bounds}

We now return from Delsarte, Schoenberg, and kissing numbers, back to
packing densities for Euclidean spaces. Recall Theorem~\ref{TKL} by
Kabatiansky--Levenshtein, which involved using linear programming to
obtain an upper bound on the Kissing number/spherical code $A(n,\pi/3)$.
From this, the authors deduced an upper bound on the packing density
itself -- which remained the state-of-the-art for many years:

\begin{theorem}[\cite{KL}]
For all angles $\theta \in [\pi/3,\pi]$, we have the relation
\begin{equation}
\Delta_{\R^n} \leq (1 - \cos \theta)^{n/2} 2^{-n/2} A(n+1,\theta) =
\sin(\theta/2)^n A(n+1,\theta).
\end{equation}
In particular, for $\theta = \pi/3$, we have using Theorem~\ref{TKL}:
\begin{equation}
\Delta_{\R^n} \leq 2^{-n(0.599+o(1))}.
\end{equation}
\end{theorem}

\begin{proof}[Sketch of proof, taken from~\cite{CZ}]
Let $\mathcal{P}$ be any packing of $\R^n$ with unit spheres, with
density $\Delta$. Consider a solid sphere $\overline{B}_{\R^{n+1}}(0, R)
\subset \R^{n+1}$, whose radius $R$ we choose later. Also fix a
hyperplane $P$ passing through the origin, and consider the
$n$-dimensional closed disk $D := P \cap \overline{B}_{\R^{n+1}}(0, R)$.
Given that $D$ has $n$-volume equal to that of $R^n$ unit spheres, one
expects that on average, a translation of the above sphere packing of
$\R^n \cong P$ should intersect $D$ in $\Delta R^n$ many \textit{centers}
of unit spheres. Choose such a translation, and project these centers
onto the upper hemisphere of the boundary $R \cdot S^n$. The spherical
distance between any two of these points is bigger than their Euclidean
distance, which in turn exceeds the distance between their projections in
$D$.

Now we bring in $\theta$. Ensuring that any two projected points in the
hemisphere are at least $\theta$ angle apart, is equivalent to their arc
having length $\geq R \theta$, which holds if the chord joining them has
length $\geq 2 R \sin(\theta/2)$. But we know that all chords have length
at least $2$, so we equate these bounds and set $R = 1/\sin(\theta/2) >
1/(\theta/2)$. Then
\[
\Delta \cdot R^n \leq A(n+1,\theta),
\]
and this is the desired bound.
\end{proof}

In 2014, Cohn and Zhao improved this bound, using a similarly short
argument:

\begin{theorem}[\cite{CZ}]
For all angles $\theta \in [\pi/3,\pi]$, we have the relation
\begin{equation}
\Delta_{\R^n} \leq \sin(\theta/2)^n A(n,\theta).
\end{equation}
\end{theorem}

This is at least as good, because all $A(n,\theta)$-many points can be
arranged along the equatorial sub-sphere $S^{n-1}$ in $S^n \subset
\R^{n+1}$.

We now list a few previously shown upper bounds. The first is by Rogers
in 1958:

\begin{theorem}[\cite{Rogers58}]
Let $X$ be a regular $(n+1)$-point simplex in $\R^n$ of side length $2$
and a vertex at the origin. Then
\begin{equation}\label{Erogers58}
\Delta_{\R^n}  \leq \sigma_n, \qquad \text{where } \sigma_n :=
\frac{(n+1) {\rm Vol}_{\R^n}(B_{\R^n}(0,1) \cap X)}{{\rm Vol}_{\R^n}(X)}.
\end{equation}
\end{theorem}

In other words, the packing density of $\R^n$ cannot exceed the packing
density of a regular simplex of edgelength $2$ with unit spheres at the
vertices.

Our final two bounds here are older: in 1929, Blichfeldt showed
in~\cite{Bl29} that $\Delta_{\R^n} \leq \frac{n+2}{2} \cdot 2^{-n/2}$, 
and in 1944, Mordell~\cite{Mordell2} provided a bound in the language of
the Hermite constant:
\begin{equation}\label{Emordell}
\gamma_n \leq \gamma_{n-1}^{(n-1)/(n-2)}.
\end{equation}

\begin{remark}
Repeated use of Mordell's bound~\eqref{Emordell} shows that
\[
\gamma_n \leq \gamma_{n-1}^{\frac{n-1}{n-2}} \leq
\gamma_{n-2}^{\frac{n-1}{n-3}} \leq \cdots \leq \gamma_2^{n-1}.
\]
But Lagrange's theorem~\ref{Tlagrange} shows $\gamma_2 = \sqrt{4/3}$; now
combining these two facts yields Hermite's original
theorem~\ref{Thermite}, reformulated as: $\gamma_n \leq (4/3)^{(n-1)/2} =
\gamma_2^{n-1}$.
\end{remark}

We end this part by mentioning that work on upper bounds continues -- see
the very recent work~\cite{SZ}, where new upper bounds are shown both for
spherical codes (for angles $\theta < 62.997^\circ$), and then for sphere
packings in dimensions $n \geq 2000$. (This work improves by a constant
factor the Kabatiansky--Levenshtein upper bound~\cite{KL}, as did
Cohn--Zhao~\cite{CZ}.) We add that Cohn maintains a
\href{https://cohn.mit.edu/sphere-packing/}{webpage} with the latest
numerical upper bounds in low dimensions, and references for these.

\subsection{Lower bounds on sphere packings}

In comparison, there are many results in the 20th and 21st centuries that
address lower bounds for the packing density $\Delta_{\R^n}$. The first
is a simple ``folklore'' estimate.

\begin{lemma}
$\Delta_{\R^n} \geq 2^{-n}$.
\end{lemma}

\begin{proof}
In fact, we claim that this estimate is achieved by any ``saturated''
packing $\mathcal{P}$ -- one in which no additional sphere can be
inserted. To see why, note as above that all spheres in a packing have
centers separated by a distance of $2$. Now we claim that $\bigcup_{x \in
\mathcal{P}} B_{\R^n}(x,2) = \R^n$ -- for if not, there would be a point
in the complement, which is $2$ apart from all other centers. But then
one can add another unit sphere here, which contradicts saturation. Now
since the density of the ``doubled spheres'' is $1$, the result follows.
\end{proof}

This bound has seen several improvements throughout the past century and
this one -- many of the following are lower bounds even for lattice
sphere packing densities.

\begin{theorem}
The packing density $\Delta_{\R^n} \cdot 2^n$ is at least as large as:
\begin{enumerate}
\item (Minkowski, 1905 \cite{Minkowski} and Hlawka, 1943 \cite{Hlawka}.)
$2 \zeta(n) = 2 \sum_{j=1}^\infty j^{-n} = 2 + O(2^{-n})$.

\item (Rogers, 1947 \cite{Rogers47}.)
$\displaystyle \frac{2n \zeta(n)}{e(1-e^{-n})}$. In particular, this
exceeds $0.73n$ for $n \gg 0$. (See also Davenport and
Rogers~\cite{DR47}.)

\item (Ball, 1992 \cite{Ball}.)
$2(n-1) \zeta(n)$.

\item (Krivelevich--Litsyn--Vardy, 2004 \cite{KLV}.)
$n/100$.

\item (Vance, 2011 \cite{Vance}.)
$\displaystyle \frac{6n \zeta(n)}{e(1-e^{-n/4})}$, if $4|n$.

\item (Venkatesh, 2012 \cite{Venkatesh}.)
$65963 n$ for all sufficiently large $n$; and (the first super-linear
growth:) $\frac{1}{2} n \log \log n$ for infinitely many $n$.

\item (Campos--Jenssen--Michelen--Sahasrabudhe \cite{CJMS}.)
$\frac{1 - o(1)}{2} n \log n$.

\item (Gargava--Viazovska \cite{GV}.)
$n \log \log n - O(e^{-c n(\log n + O(1))}$ for infinitely many $n$,
where $c>0$ is a universal constant.

\item (Klartag \cite{Klartag}.)
$c n^2$, where $c>0$ is a universal constant. (In a sense, this is an
adaptation/follow-up of the work of Rogers.)
\end{enumerate}
\end{theorem}

Thus, the search for better -- maybe even sharp -- lower and upper bounds
for sphere packings in \textit{general} (and large) dimension $n$ is by
no means over, and should see more exciting developments in the years
ahead.

\subsection{Conclusion: Cohn--Elkies and Viazovska}

Having discussed (asymptotic) upper and lower bounds for the packing
density in all/large dimensions, as well as low-dimensional special cases
$(n=1,2,3$), we now conclude this section by mentioning the recent
success in determining $\Delta_{\R^n}$ for $n=8, 24$ -- once again using
the lattices $E_8$ and $\Lambda_{24}$, respectively. In these specific
dimensions, one uses linear programming bounds once again -- not on
spherical codes via Delsarte's methods, but directly on $\R^n$ via a 2003
result of Cohn and Elkies. To state it, we first recall that the
\textit{Fourier transform} of an $L^1$ map $f : \R^n \to \R$ is:
\[
\widehat{f}(y) := \int_{\R^n} f(x) e^{-2 \pi i \tangle{x,y}}\ dx.
\]

We also need the following notion.

\begin{definition}
An $L^1$ map $f : \R^n \to \R$ is said to be \textit{admissible} if there
exists $\delta \in (0,\infty)$ such that $|f(x)|$ and $|\widehat{f}(x)|$
are bounded above by a constant times $(1 + \|x\|)^{-n-\delta}$.
\end{definition}

Now Cohn and Elkies show:

\begin{theorem}[\cite{CE}]\label{TCE}
Suppose $f : \R^n \to \R$ is admissible, and $r>0$ is such that
\begin{enumerate}
\item $f(0), \widehat{f}(0) > 0$,
\item $f(x) \leq 0$ whenever $\|x\| \geq r$, and
\item $\widehat{f}(y) \geq 0$ for all $y \in \R^n$.
\end{enumerate}
Then one can upper bound the packing density of $\R^n$:
\begin{equation}\label{ECE}
\Delta_{\R^n} \cdot \frac{\widehat{f}(0)}{f(0)} \quad \leq \quad {\rm
Vol}_{\R^n}(B_{\R^n}(0,r/2)) = (r/2)^n \nu_n = (r/2)^n
\frac{\pi^{n/2}}{\Gamma(n/2+1)}.
\end{equation}

If moreover $\lambda_1(L) = r$ for some lattice $L$, then it attains the
packing bound,
\begin{equation}\label{Emagic}
\Delta_L = \Delta_{\R^n} = {\rm Vol}_{\R^n}(B_{\R^n}(0,r/2))
\frac{f(0)}{\widehat{f}(0)}
\end{equation}
if and only if $f \equiv 0$ on $L \setminus \{ 0 \}$ and
$\widehat{f} \equiv 0$ on $L^* \setminus \{ 0 \}$.
\end{theorem}

Recall here that the \textit{dual lattice} $L^*$ is defined to be the
lattice either generated by the dual basis to a given basis of $L$, or
equivalently,
\[
L^* = \{ t \in \R^n \, : \, \tangle{x,t} \in \mathbb{Z}\ \forall x \in L
\}.
\]

\begin{remark}
Note that the volume on the right of \eqref{ECE} may be $\geq 1$, in
which case the bound is trivial. However, for fixed $r$ the $n$-volume
decreases to $0^+$ as $n \to \infty$, and so the bound is certainly
relevant for large $n$.
\end{remark}

The proof of Theorem~\ref{TCE} and the subsequent analysis use the
Poisson summation formula. Note that if $f$ is admissible, then $f,
\widehat{f}$ are both in $L^1$ and continuous. Now we have:

\begin{theorem}[Poisson summation]
Suppose $f : \R^n \to \R$ is admissible, and $L \subset \R^n$ is a
lattice. Then for every $v \in \R^n$,
\begin{equation}\label{Epoisson}
\sum_{x \in L} f(x+v) = \frac{1}{{\rm Vol}_{\R^n}(\R^n/L)} \sum_{t \in
L^*} e^{-2\pi i \tangle{v,t}} \widehat{f}(t),
\end{equation}
with both sides converging absolutely.
\end{theorem}

\begin{proof}[Sketch of proof of Theorem~\ref{TCE} for lattice packings,
taken from \cite{Cohn}]
Let $L$ be any lattice in $\R^n$. Since any lattice packing of
$L$ by spheres of packing radius $r(L) = \lambda_1(L)/2$ has the same
packing density under rescaling the lattice and the spheres, let us
rescale $L$ such that $r(L) = r/2$. Then the lattice packing density is
\[
\Delta_L = \frac{{\rm Vol}_{\R^n}(B_{\R^n}(0,r/2))}{{\rm
Vol}_{\R^n}(\R^n/L)},
\]
so it suffices to show the \textbf{claim} that if $r(L) = r/2$ then $L$
has covolume at least $\widehat{f}(0)/f(0)$. To see this, apply Poisson
summation~\eqref{Epoisson} with $v=0$. As $\lambda_1(L) = r$, the
hypotheses give us that
(a)~the left side of~\eqref{Epoisson} is bounded above by $f(0)$, while
(b)~the right side is bounded below by the ratio of $\widehat{f}(0)$ and
the covolume of $L$.
This proves the claim.

For the final assertion, the preceding paragraph says that
\begin{equation}\label{Efinal}
f(0) \geq \sum_{x \in L} f(x) = \frac{1}{{\rm Vol}_{\R^n}(\R^n/L)}
\sum_{t \in L^*} \widehat{f}(t) \geq \frac{\widehat{f}(0)}{{\rm
Vol}_{\R^n}(\R^n/L)}.
\end{equation}
Hence (by the previous working,) \eqref{Emagic} holds if and only if the
extremal terms in~\eqref{Efinal} are equal, and this happens if and only
if both inequalities in~\eqref{Efinal} are equalities. This completes the
proof.
\end{proof}

This proof is remarkable, in that one ends up throwing out all but one
terms on both sides of the Poisson summation formula! So if this approach
is to yield a lattice that attains the packing density $\Delta_{\R^n}$
(and hence $\Delta_{\R^n}^{(L)}$), then a remarkably constrained function
$f$ would need to exist with all of the above properties. In~\cite{Cohn},
Cohn calls such an $f$ a \textit{magic function}.

Let us now get even \textit{more} restrictive. If the lattice achieving
this result is to be one of the two special lattices -- $L = E_8,
\Lambda_{24}$ -- then we also note that $L^* \cong L$! Moreover, $r =
\lambda_1(L)$, which is $\sqrt{2}$ for $E_8$ and $2$ for the Leech
lattice $\Lambda_{24}$, as mentioned in Section~\ref{Slattices}.

Thus, we would then need a magic function $f$ which satisfies the
hypotheses of Theorem~\ref{TCE} and all other constraints above; and
moreover, $f$ and $\widehat{f}$ vanish on all nonzero points in $E_8$ or
$\Lambda_{24}$. (Finding such an $f$ seems truly ``magical''!) Otherwise,
one has to try a completely different approach to attaining the packing
density.

Initial investigations did seem to be promising. For instance,
Cohn--Kumar showed~\cite{CK09} that in 24 dimensions, the packing density
is very close to the lattice packing density:
\[
\Delta_{\R^{24}}^{(L)} = \Delta_{\Lambda_{24}} \leq \Delta_{\R^{24}} \leq
\Delta_{\Lambda_{24}} \cdot (1 + 1.65 \cdot 10^{-30}).
\]

So, does such a magic function exist? The first answer came for $n=8$ --
in which case, it does exist! This was the main result of
Viazovska~\cite{Viazovska}, which she announced in an arXiv preprint on
``Pi Day 2016''.\footnote{Coincidentally, this date abbreviates to
$3/14/16$ -- perhaps the best date approximating $\pi$ this century, as
it gives $\pi$ rounded to four decimal places!}
The proof and the underlying magic function quickly got understood, and
within a week, Viazovska along with Cohn--Kumar--Miller--Radchenko posted
another paper~\cite{CKMRV1} where they found the analogous magic function
for the Leech lattice as well. These magic functions show:

\begin{theorem}[\cite{Viazovska,CKMRV1}]
For $n=8$, the packing density and lattice packing density agree, and
equal $\pi^4 / 384$ at the lattice $E_8$. For $n=24$, the same statement
holds, except that the lattice is $\Lambda_{24}$ and the density is
$\pi^{12} / 12!$.
\end{theorem}

In these works, Viazovska (et al) used the theory of modular forms to
come up with the relevant magic functions. This was followed by the 2022 work of
Cohn--Kumar--Miller--Radchenko--Viazovska \cite{CKMRV2}, where they also
showed that the lattices $E_8$ and $\Lambda_{24}$ in fact minimized
energy for every potential function that is a completely monotonic
function of $\|x\|^2$ -- such a phenomenon was previously known to hold
only for $n=1$. In particular, this also generalizes the optimality of
the sphere packing densities on these lattices. For additional background
and details, we refer the reader to two articles by Cohn. The first is
his beautiful account~\cite{Cohn} of Viazovska's work and its recent
predecessors. The second consists of his lecture
notes~\cite{Cohn-lectures}, which address not only sphere packing and
kissing numbers, but also spherical harmonics, energy minimization, and
more.

\subsection*{Acknowledgments}

I firstly thank Gadadhar Misra for asking me to write a survey on
Schoenberg's theorem, and also for his patience during my writing the
survey and then expanding it to include the Appendix.
I thank Mihai Putinar for many enriching additions and references,
E.K.~Narayanan for useful conversations about spherical harmonics,
and Dominique Guillot for going through an early draft in detail and
offering numerous helpful suggestions that improved the exposition.
I thank Sujit Sakharam Damase for doing so as well, and for valuable
discussions about Sections~\ref{Sposdef}, \ref{SDP}, and
Appendix~\ref{Asphere}. In particular, the Appendix would not have been
written without his enthusiasm for sphere packings and Gegenbauer (and
other orthogonal families of) polynomials, perhaps owing to his work with
James Pascoe~\cite{DP}.

It is also worth pointing out that both positivity and its preservers, as
well as sphere packings, are areas of mathematics with numerous
connections to other fields and extensive research over more than a
century, and to do justice to either in a few pages is not possible or
even reasonable. In particular, any omissions and errors are mine.



\end{document}